\documentclass[reqno]{amsart}

\usepackage{graphicx}
\usepackage{amsrefs}
\usepackage{amsmath,amsthm,amssymb,amscd,latexsym}
\usepackage{tikz}
\usepackage{multirow}
\usepackage[colorlinks]{hyperref}
\hypersetup{colorlinks,linkcolor=black,filecolor=black,urlcolor=black,citecolor=black}

\newtheorem{theorem}{Theorem}
\newtheorem{corollary}[theorem]{Corollary}

\newtheorem{fact}[theorem]{Fact}

\newtheorem{observation}[theorem]{Observation}
\newtheorem{proposition}[theorem]{Proposition}
\newtheorem{lemma}[theorem]{Lemma}
\theoremstyle{definition}
\newtheorem{remark}[theorem]{Remark}
\newtheorem{question}{Problem}
\newtheorem{definition}[theorem]{Definition}

{
\theoremstyle{remark}

\newtheorem*{claim*}{Claim}
}
{
\theoremstyle{definition}
}

\numberwithin{theorem}{section}
\numberwithin{equation}{section}

\newcommand{\st}{\mbox{ : }}
\newcommand{\stm}{\setminus}
\newcommand{\sbs}{\subset}

\newcommand{\old}[1]{}
\newcommand{\arxiv}[1]{{\tt \href{http://arxiv.org/abs/#1}{arXiv:#1}}}

\newcommand{\rst}[1]{\ensuremath{{\mathbin\upharpoonright}%
\raise-.5ex\hbox{$#1$}}}  

\newcommand{\tri}{\triangle}

\renewcommand{\Re}[1]{\mbox{Re}(#1)}
\renewcommand{\Im}[1]{\mbox{Im}(#1)}

\newcommand{\aaa}{\mathcal{A}}
\newcommand{\ppp}{\mathcal{P}}
\newcommand{\bbb}{\mathcal{B}}
\newcommand{\ccc}{\mathcal{C}}
\newcommand{\sss}{\mathcal{S}}
\newcommand{\ttt}{\mathcal{T}}
\newcommand{\LLL}{\mathcal{L}}

\newcommand{\R}{\mathbb{R}}
\newcommand{\C}{\mathbb{C}}
\newcommand{\Z}{\mathbb{Z}}

\newcommand{\N}{\mathbb{N}}
\newcommand{\ep}{\varepsilon}
\newcommand{\trace}{\operatorname{tr}}
\newcommand{\abs}[1]{\left|#1\right|}

\newcommand{\clg}[1]{\left\lceil#1\right\rceil}

\renewcommand{\phi}{\varphi}

\title{Apollonian structure in the Abelian sandpile}

\author{Lionel Levine}
\address{Department of Mathematics, Cornell University, Ithaca, NY 14853. \url{http://www.math.cornell.edu/~levine}}

\author{Wesley Pegden}
\address{Carnegie Mellon University, Pittsburgh PA 15213.}
\email{wes@math.cmu.edu}

\author{Charles K. Smart}
\address{Massachusetts Institute of Technology, Cambridge, MA 02139}
\email{smart@math.mit.edu}

\date{May 22, 2014}
\thanks{The authors were partially supported by NSF grants DMS-1004696, DMS-1004595 and DMS-1243606.}

\begin{document}

\begin{abstract}
The \emph{Abelian sandpile} process evolves configurations of chips on the integer lattice by \emph{toppling} any vertex with at least 4 chips, distributing one of its chips to each of its 4 neighbors.  When begun from a large stack of chips, the terminal state of the sandpile has a curious fractal structure which has remained unexplained.  Using a characterization of the quadratic growths attainable by integer-superharmonic functions, we prove that the \emph{sandpile PDE} recently shown to characterize the scaling limit of the sandpile admits certain fractal solutions, giving a precise mathematical perspective on the fractal nature of the sandpile.
\end{abstract}

\maketitle

\section{Introduction}

\subsection{Background} First introduced in 1987 by Bak, Tang and Wiesenfeld \cite{Bak-Tang-Wiesenfeld} as a model of self-organized criticality, the Abelian sandpile is an elegant example of a simple rule producing surprising complexity. In its simplest form, the sandpile evolves a configuration  $\eta : \Z^2 \to \N$ of {\em chips} by iterating a simple process: find a lattice point $x \in \Z^2$ with at least four chips and {\em topple} it, moving one chip from $x$ to each of its four lattice neighbors.

When the initial configuration has finitely many total chips, the sandpile process always finds a {\em stable} configuration, where each lattice point has at most three chips.  Dhar \cite{Dhar} observed that the resulting stable configuration does not depend on the toppling order, which is the reason for terming the process ``Abelian.''  When the initial configuration consists of a large number of chips at the origin, the final configuration has a curious fractal structure \cite{Liu-Kaplan-Gray,Ostojic,Dhar-Sadhu-Chandra,strings,pathesis} which (after rescaling) is insensitive to the number of chips. In 25 years of research (see \cite{Levine-Propp} for a brief survey, and \cite{Dhar06,Redig} for more detail)
this fractal structure has resisted explanation or even a precise description.  

If $s_n : \Z^2 \to \N$ denotes the stabilization of $n$ chips placed at the origin, then the rescaled configurations
\begin{equation*}
\bar s_n(x) := s_n([n^{1/2} x])
\end{equation*}
(where $[x]$ indicates a closest lattice point to $x\in \R^2$) converge to a unique limit $s_\infty$.  This article presents a partial explanation for the apparent fractal structure of this limit.

The convergence $\bar s_n \to s_\infty$ was obtained Pegden-Smart \cite{charlie-wes}, who used viscosity solution theory to identify the continuum limit of the least action principle of Fey-Levine-Peres \cite{Fey-Levine-Peres}.  We call a $2 \times 2$ real symmetric matrix $A$ {\em stabilizable} if there is a function $u : \Z^2 \to \Z$ such that
\begin{equation}
\label{l.gamma}
u(x) \geq \frac{1}{2} x^t A x \quad \mbox{and} \quad \Delta^1 u(x) \leq 3,
\end{equation}
for all $x \in \Z^2$, where
\begin{equation}
\Delta^1 u(x)=\sum_{y \sim x}(u(y)-u(x))
\label{l.lap}
\end{equation}
is the discrete Laplacian of $u$ on $\Z^2$. (We establish a direct correspondence between stabilizable matrices and infinite stabilizable sandpile configurations in Section \ref{s.gamma}.)  It turns out that the closure $\bar \Gamma$ of the set $\Gamma$ of stabilizable matrices determines $s_\infty$.

\begin{figure}[t]
\centering
\includegraphics[width=.95\linewidth]{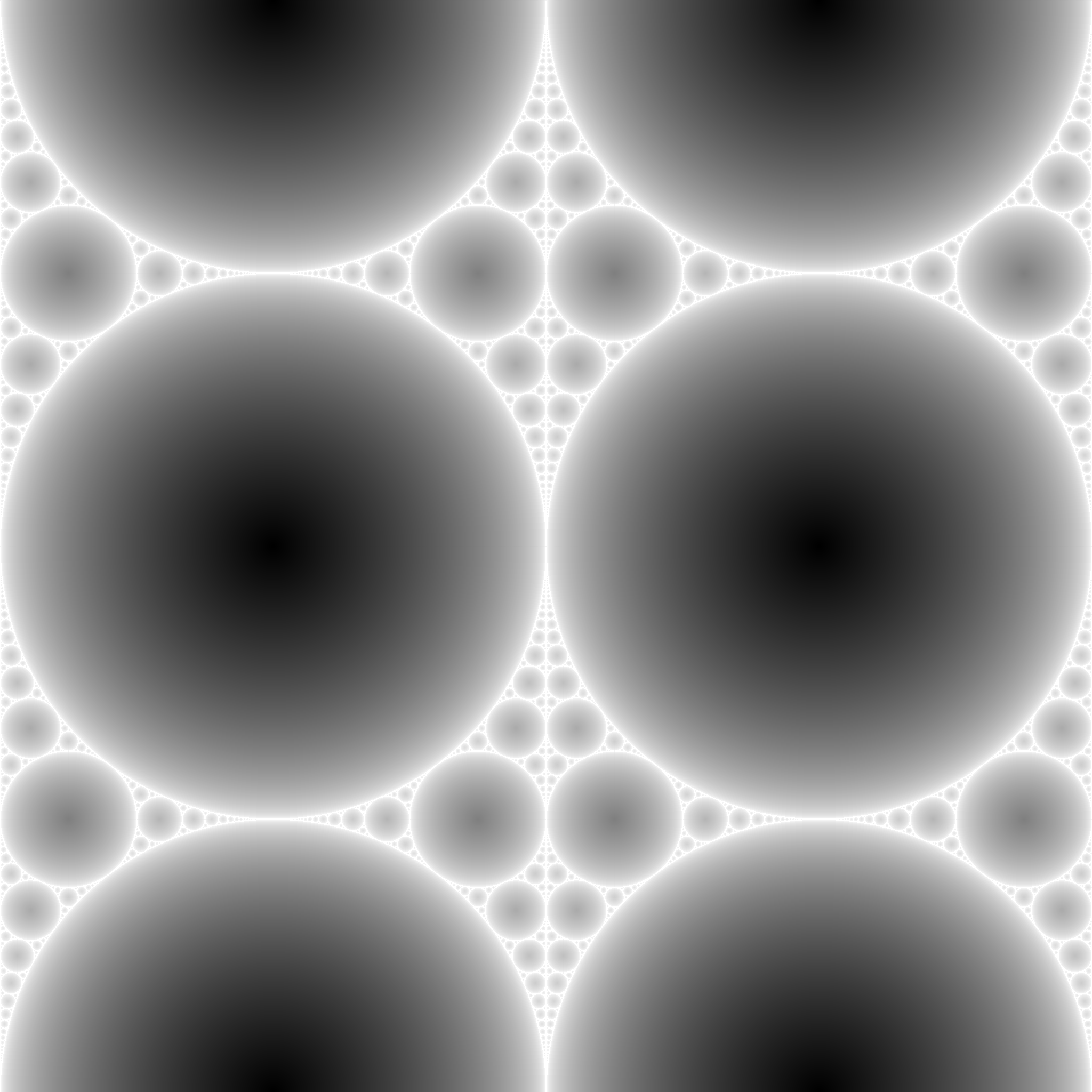}
\caption{\label{f.numerical}
The boundary of $\Gamma$. The shade of gray at location $(a,b) \in [0,4]\times [0,4]$ indicates the largest $c \in [2,3]$ such that $M(a,b,c) \in \Gamma$.  White and black correspond to $c=2$ and $c=3$, respectively.
}
\end{figure}

\begin{theorem}[Existence Of Scaling Limit, \cite{charlie-wes}]
\label{t.charlie-wes}
The rescaled configurations $\bar s_n$ converge weakly-$*$ in $L^\infty(\R^2)$ to $s_\infty = \Delta v_\infty$, where
\begin{equation}
\label{obstacle}
v_\infty := \min \{ w \in C(\R^2) \mid w \geq - \Phi \mbox{ and } D^2 (w + \Phi) \in \bar \Gamma \}.
\end{equation}
Here $\Phi(x) := - (2 \pi)^{-1} \log |x|$ is the fundamental solution of the Laplace equation $\Delta \Phi=0$, the minimum is taken pointwise, and the differential inclusion is interpreted in the sense of viscosity.
\end{theorem}

Roughly speaking, the sum $u_\infty = v_\infty + \Phi$ is the least function $u \in C(\R^2\stm \{0\})$ that is non-negative, grows like $\Phi$ at the origin, and solves the {\em sandpile PDE}
\begin{equation}
\label{l.sandpde}
D^2 u  \in \partial \Gamma
\end{equation}
in $\{ u > 0 \}$ in the sense of viscosity. Our use of viscosity solutions is described in more detail in the preliminaries; see Section \ref{s.viscosity}. The function $u_\infty$ also has a natural interpretation in terms of the sandpile: it is the limit $u_\infty(x) = \lim_{n \to \infty} n^{-1} u_n([n^{1/2} x])$, where $u_n(x)$ is the number of times $x \in \Z^d$ topples during the formation of $s_n$.  We also recall that weak-$*$ convergence simply captures convergence of the local average value of $\bar s_n$.

\subsection{Apollonian structure} 

The key players in the obstacle problem \eqref{obstacle} are $\Phi$ and $\Gamma$.  The former encodes the initial condition (with the particular choice of $-(2\pi)^{-1} \log |x|$ corresponding to all particles starting at the origin).  The set~$\Gamma$ is a more interesting object: it encodes the continuum limit of the sandpile stabilization rule.
It turns out that $\bar \Gamma$ is a union of downward cones based at points of a certain set $\mathcal{P}$---this is Theorem \ref{t.apollo}, below, which we prove in the companion paper \cite{2014integersuperharmonic}.  The elements of $\mathcal{P}$, which we call \emph{peaks}, are visible as the locally darkest points in Figure \ref{f.numerical}. 

The characterization of  $\bar \Gamma$ is made in terms of Apollonian configurations of circles.  Three pairwise externally tangent circles $C_1,C_2,C_3$ determine an \emph{Apollonian circle packing}, as the smallest set of circles containing them that is closed under the operation of adding, for each pairwise tangent triple of circles, the two circles which are tangent to each circle in the triple.  They also determine a \emph{downward} \emph{Apollonian packing}, closed under adding, for each pairwise-tangent triple, only the smaller of the two tangent circles.   Lines are allowed as circles, and the \emph{Apollonian band circle packing} is the packing $\bbb_0$ determined by the lines $\{ x = 0 \}$ and $\{ x = 2 \}$ and the circle $\{ (x-1)^2 + y^2 = 1 \}$.  Its circles are all contained in the strip $[0,2] \times \R$.

We put the proper circles in $\R^2$ (i.e., the circles that are not lines) in bijective correspondence with real symmetric $2\times 2$ matrices of trace $>2$, in the following way.  To a proper circle $C = \{(x-a)^2 + (y-b)^2 = r^2 \}$ in $\R^2$ we associate the matrix
	\[ m(C) := M(a,b,r+2) \]
where
	\begin{equation}
	\label{l.mparam}
	M(a,b,c):= \frac12 \left[ \begin{matrix} c + a & b \\ b & c - a \end{matrix} \right].
	\end{equation}
We write $S_2$ for the set of symmetric $2 \times 2$ matrices with real entries, and, for $A,B \in S_2$ we write $B \leq A$ if $A-B$ is nonnegative definite.  For a set $\ppp \subset S_2$, we define 
	\[ \ppp^\downarrow := \{B \in S_2 \mid B \leq A \mbox{ for some } A \in \ppp \}, \]
the order ideal generated by $\ppp$ in the matrix order.

Now let $\bbb = \bigcup_{k \in \Z} (\bbb_0 + (2k,0))$ be the extension of the Apollonian band packing to all of $\R^2$ by translation.  Let
	\[ \ppp = \{ m(C) \mid C \in \bbb \}. \]
In the companion paper \cite{2014integersuperharmonic}, a function $g_A:\Z^2\to \Z$ with $\Delta^1 g_A\leq 1$  is constructed for each $A$ such that $A+M(2,0,2)\in \ppp$ whose difference from $\frac 1 2 x^t A x+b_A\cdot x$ is periodic and thus at most a constant, for some linear factor $b_A$.  Moreover, the functions $g_A$ are maximally stable, in the sense that $g\geq g_A$ and $\Delta^1 g\leq 1$ implies that $g-g_A$ is bounded.  By adding $x_1^2$ to each such $g_A$ and $M(2,0,2)$ to each $A$, this construction from \cite{2014integersuperharmonic} gives the following theorem:
\begin{theorem}[\cite{2014integersuperharmonic}]
\label{t.apollo}
$ \bar \Gamma = \ppp^\downarrow. $ \qed
\end{theorem}


\begin{figure}[t]
\centering
\includegraphics[angle=-90, width=.7\linewidth]{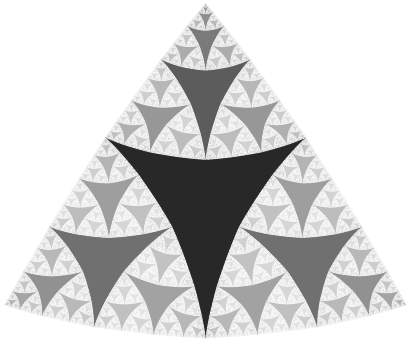}
\caption{\label{f.sandfractal} An Apollonian triangulation is a union of Apollonian triangles meeting at right angles, whose intersection structure matches the tangency structure of their corresponding circles.  The solution $u$ of Theorem~\ref{t.mainpq} has constant Laplacian on each Apollonian triangle, as indicated by the shading (darker regions are where $\Delta u$ is larger).}
\end{figure}
 
\subsection{The sandpile PDE}
Theorem~\ref{t.apollo} allows us to formulate the sandpile PDE \eqref{l.sandpde} as
	\begin{equation} \label{e.apollopde} D^2 u \in \partial \ppp^\downarrow. \end{equation}
Our main result, Theorem~\ref{t.mainpq} below, constructs a family of piecewise quadratic solutions to the this PDE.  The supports of these solutions are the closures of certain fractal subsets of $\R^2$ which we call \emph{Apollonian triangulations}, giving an explanation for the fractal limit $\bar s_\infty$.

Of course, every matrix $A =M(a,b,c) \in S_2$ with $\trace(A) = c >2$ is now associated to a unique proper circle $C = c(A) = m^{-1}(A)$ in $\R^2$. We say two matrices are (\emph{externally}) \emph{tangent} precisely if their corresponding circles are (externally) tangent.   Given pairwise externally tangent matrices $A_1,A_2,A_3$, denote by $\aaa(A_1,A_2,A_3)$ (resp.~$\aaa^-(A_1,A_2,A_3)$) the set of matrices corresponding to the Apollonian circle packing (resp.~downward Apollonian packing) determined by the circles corresponding to $A_1,A_2,A_3$.

\begin{theorem}
[Piecewise Quadratic Solutions]
\label{t.mainpq}
For any pairwise externally tangent matrices $A_1,A_2,A_3 \in S_2$, there is a nonempty convex set $Z \subset \R^2$ and a function $u \in C^{1,1}(Z)$ satisfying \[ D^2 u \in \partial \aaa(A_1,A_2,A_3)^\downarrow \] in the sense of viscosity.  Moreover, $Z$ decomposes into disjoint open sets (whose closures cover $Z$) on each of which $u$ is quadratic with Hessian in $\aaa^-(A_1,A_2,A_3)$.
\end{theorem}

This theorem is illustrated in Figure~\ref{f.sandfractal}.  We call the configuration of pieces where $D^2u$ is constant 
an \emph{Apollonian triangulation}.  Our geometric characterization of Apollonian triangulations begins with the definition of {\em Apollonian curves} and {\em Apollonian triangles} in Section \ref{s.geo}. We will see that three vertices in general position determine a unique Apollonian triangle with those vertices, via a purely geometric construction based on medians of triangles.  We will also show that any Apollonian triangle occupies exactly $4/7$ of the area of the Euclidean triangle with the same vertices. 

 An Apollonian triangulation, which we precisely define in Section \ref{s.arec}, is a union of Apollonian triangles corresponding to circles in an Apollonian circle packing, where pairs of Apollonian triangles corresponding to pairs of intersecting circles meet at right angles. The existence of Apollonian triangulations is itself nontrivial and is the subject of Theorem \ref{t.sandfrac}; analogous discrete structures were constructed by Paoletti in his thesis \cite{pathesis}.     Looking at the Apollonian fractal in Figure \ref{f.sandfractal} and recalling the $SL_2(\Z)$ symmetries of Apollonian circle packings, it is natural to wonder whether nice symmetries may relate distinct Apollonian triangulations as well.  But we will see in Section \ref{s.arec} that Apollonian triangles are equivalent under affine transformations, precluding the possibility of conformal equivalence for Apollonian triangulations. 

\begin{figure}[t]
\begin{center}
\includegraphics[height=6.2cm]{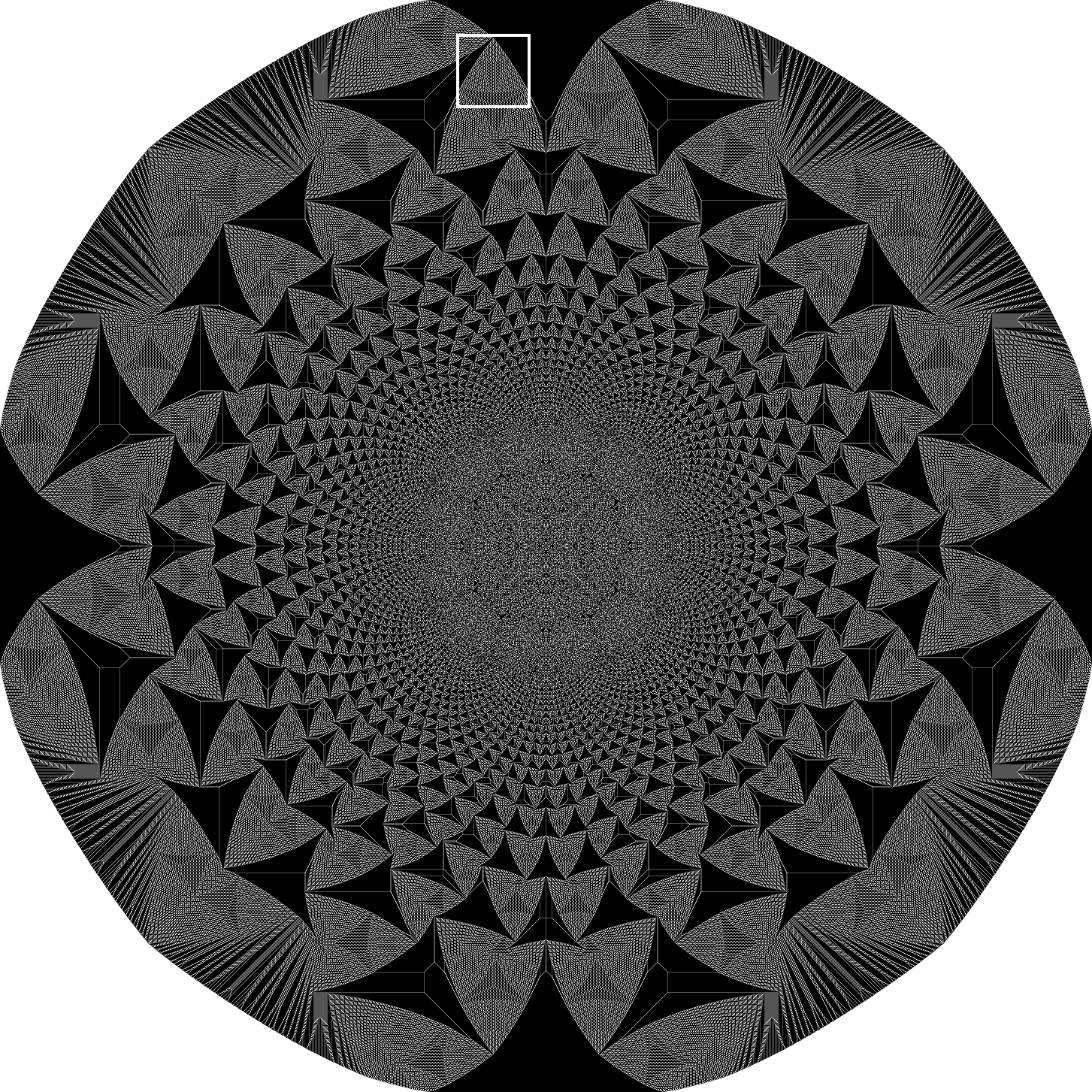}
\includegraphics[height=6.2cm]{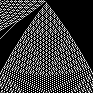}\hspace{.3cm}
\end{center}
\caption{Left: The sandpile $s_n$ for $n = 4\cdot 10^6$. Sites with 0, 1, 2, and 3 chips are represented by 
four different shades of gray.  Right: A zoomed view of the boxed region, one of many that we believe converges to an Apollonian triangulation in the $n\to \infty$ limit.}
\label{f.sandsequence}
\end{figure}

If $C_1,C_2,C_3$ are pairwise tangent circles in the band circle packing, then letting $A_i = m(C_i)$ for $i=1,2,3$, we have $\aaa^-(A_1,A_2,A_3) \subset \ppp$, so the function $u$ in Theorem~\ref{t.mainpq} will be a viscosity solution to the sandpile PDE.  The uniqueness machinery for viscosity solutions gives the following corollary to Theorem \ref{t.mainpq}, which encapsulates its relevance to the Abelian sandpile.

 \begin{corollary}
\label{c.match}
Suppose $U_1, U_2, U_3 \subseteq \R^2$ are connected open sets bounding a convex region $Z$ such that $\bar U_i \cap \bar U_j = \{ x_k \}$ for $\{ i, j, k \} = \{ 1, 2, 3 \}$, where the triangle $\tri x_1x_2x_3$ is acute. If $u_\infty$ is quadratic on each of $U_1, U_2, U_3$ with pairwise tangent Hessians $A_1, A_2, A_3 \in \ppp$, respectively, then $u_\infty$ is piecewise quadratic in $R$ and the domains of the quadratic pieces form the Apollonian triangulation determined by the vertices $x_1, x_2, x_3$.
 \end{corollary}
\noindent Note that $s_\infty=\Delta v_\infty$ implies $\bar s_\infty$ is piecewise-constant in the Apollonian triangulation.

Let us briefly remark on the consequences of this corollary for our understanding of the limit sandpile.  As observed in \cite{Ostojic,Dhar-Sadhu-Chandra} and visible in Figure \ref{f.sandsequence}, the sandpile $s_n$ for large $n$ features many clearly visible patches, each with its own characteristic periodic pattern of sand (sometimes punctuated by one-dimensional `defects' which are not relevant to the weak-* limit of the sandpile). Empirically, we observe that triples of touching regions of these kinds are always regions where the observed finite $\bar v_n$ correspond (away from the one-dimensional defects) exactly to minimal representatives in the sense of \eqref{l.gamma} of quadratic forms 
\[
\frac 1 2 x^tAx+bx
\]
where the $A$'s for each region are always as required by Corollary \ref{c.match}.  Thus we are confident from the numerical evidence that the conditions required for Corollary \ref{c.match} and thus Apollonian triangulations occur---indeed, are nearly ubiquitous---in $s_\infty$.  Going beyond Corollary \ref{c.match}'s dependence on local boundary knowledge would seem to require an understanding the global geometry of $s_\infty$, which remains a considerable challenge.

\subsection{Overview} The rest of the paper proceeds as follows. In Section \ref{s.prelim}, we review some background material on the Abelian sandpile and viscosity solutions.  In section \ref{s.gamma}, we present an algorithm for computing $\Gamma$ numerically; this provided the first hints towards Theorem \ref{t.apollo}, and now provides the only window we have into sets analogous to $\Gamma$ on periodic graphs in the plane other than $\Z^2$ (see Question \ref{q.tri} in Section \ref{s.Q}).  After reviewing some basic geometry of Apollonian circle packings in Section \ref{s.packings}, we define and study \emph{Apollonian curves}, \emph{Apollonian triangles}, and \emph{Apollonian triangulations} in Sections \ref{s.geo} and \ref{s.arec}.  The proofs of Theorem~\ref{t.mainpq} and Corollary \ref{c.match} come in Section \ref{s.pq} where we construct piecewise-quadratic solutions to the sandpile PDE.  Finally, in Section \ref{s.Q} we discuss new problems suggested by our results.

\section{Preliminaries}
\label{s.prelim}
The preliminaries here are largely section-specific, with Section \ref{s.presand} being necessary for Section \ref{s.gamma} and Sections \ref{s.mg} and \ref{s.viscosity} being necessary for Section \ref{s.pq}.

\subsection{The Abelian sandpile}
\label{s.presand}
  Given a configuration $\eta:\Z^2\to \Z$ of chips on the integer lattice, we define a toppling sequence as a finite or infinite sequence $x_1,x_2,x_3,\ldots$ of vertices to be toppled in the sequence order, such that any vertex topples only finitely many times (thus giving a well-defined terminal configuration).  A sequence is \emph{legal} if it only topples vertices with at least 4 chips, and \emph{stabilizing} if there are at most 3 chips at every vertex in the terminal configuration.    We say that $\eta$ is \emph{stabilizable} if there exists a legal stabilizing toppling sequence.

The theory of the Abelian sandpile begins with the following standard fact:
\begin{proposition}
Any $x\in \Z^2$ topples at most as many times in any legal sequence as it does in any stabilizing sequence.\qed
\label{p.legstab}
\end{proposition}

\noindent Proposition \ref{p.legstab} implies that to any stabilizable initial configuration $\eta$, we can associate an \emph{odometer function} $v:\Z^2\to \N$ which counts the number of times each vertex topples in any legal stabilizing sequence of topplings.   The terminal configuration of any such sequence of topplings is then given by $\eta+\Delta^1 v$.  Since $v$ and so $\Delta^1 v$ are independent of the particular legal stabilizing sequence, this shows that the sandpile process is indeed ``Abelian'': if we start with some stabilizable configuration $\eta\geq 0$, and topple vertices with at least $4$ chips until we cannot do so any more, then the final configuration $\eta+\Delta^1 v$ is determined by $\eta$.

The discrete Laplacian is monotone, in the sense that $\Delta^1 u(x)$ is decreasing in $u(x)$ and increasing in $u(y)$ for any neighbor $y \sim x$ of $x$ in $\Z^2$. An obvious consequence of monotonicity is that taking a pointwise minimum of two functions cannot increase the Laplacian at a point:

\begin{proposition}
\label{p.minD}
If $u, v : \Z^d \to \Z$, $w := \min \{ u, v \}$, and $w(x) = u(x)$, then $\Delta^1 w(x) \leq \Delta^1 u(x)$.\qed
\end{proposition}

In particular, given any functions $u,v$ satisfying $\eta+\Delta^1(u)\leq 3$ and $\eta+\Delta^1(v)\leq 3$, their pointwise minimum satisfies the same constraint.  The proof of Theorem~\ref{t.charlie-wes} in \cite{charlie-wes} begins from the \emph{Least Action Principle} formulated in \cite{Fey-Levine-Peres}, which states that the odometer of an initial configuration $\eta$ is the pointwise minimum of all such functions.
\begin{proposition}[Least Action Principle]
\label{p.la}
  If $\eta:\Z^2\to \N$ and $w:\Z^2\to \N$ satisfy $\eta+\Delta^1 w\leq 3$, then $\eta$ is stabilizable, and its odometer $v$ satisfies $v\leq w$.
\end{proposition}
Note that the Least Action Principle can be deduced from Proposition \ref{p.legstab} by  associating a stabilizing sequence to $w$.  By considering the function $u=v-1$ for any odometer function $v$, the Least Action Principle implies the following proposition:

\begin{proposition}
  If $\eta:\Z^2\to \Z$ is a stabilizable configuration, then its odometer $v$ satisfies $v(x)=0$ for some $x \in \Z^2$.
\label{p.v0}
\end{proposition}

Finally, we note that these propositions generalize in a natural way from $\Z^2$ to arbitrary graphs; in our case, it is sufficient to note that they hold as well on the torus
\begin{equation*}
T_n:=\Z^2/n\Z^2\quad\mbox{for }n\in \Z^+.
\end{equation*}

\subsection{Some matrix geometry}
\label{s.mg}
All matrices considered in this paper are $2\times 2$ real symmetric matrices and we parameterize the space $S_2$ of such matrices via $M : \R^3 \to S_2$ defined in \eqref{l.mparam}. We use the usual matrix ordering: $A \leq B$ if and only if $B - A$ is nonnegative definite.

Of particular importance to us is the downward cone
\begin{equation*}
A^\downarrow := \{ B \in S_2 : B \leq A \}.
\end{equation*}
Recall that if $B \in \partial A^\downarrow$, then $A - B = v \otimes v = v v^t$ for some column vector $v$.  That is, the boundary $\partial A^\downarrow$ consists of all downward rank-$1$ perturbations of $A$.

Our choice of parameterization $M$ was chosen to make $A^\downarrow$ a cone in the usual sense.  Observe that
\begin{equation*}
M(a,b,c) \geq 0 \quad \mbox{if and only if} \quad c \geq (a^2 + b^2)^{1/2}.
\end{equation*}
Moreover:
\begin{observation}  We have
\begin{equation}
v\otimes v = M(u_1,u_2,(u_1^2+u_2^2)^{1/2})
\end{equation}
if and only if $v^2=u$, where $v^2$ denotes the complex square of $v$. \qed
\label{o.vecsq}
\end{observation}
\noindent Thus if $B \in \partial A^\downarrow$, then
\begin{equation}
A - B = (\bar \rho(B) - \bar \rho(A))^{1/2} \otimes (\bar \rho(B) - \bar \rho(A))^{1/2},
\end{equation}
where
\begin{equation*}
\bar \rho(M(a,b,c)) := (a,b),
\end{equation*}
and $v^{1/2}$ denotes the complex square root of a vector $v \in \R^2 = \C$.

Denoting by $I$ the $2\times 2$ identity matrix, we write
	\[ A^- = A - 2 (\trace(A) - 2) I \]
for the reflection of $A$ across the trace-$2$ plane; and
	\[ A^0 = \frac{A + A^-}{2} \]
for the projection of $A$ on the trace-$2$ plane. Since the line $\{A+t(v\otimes v) \mid t\in \R\}$ is tangent to the downward cone $A^\downarrow$ for every nonzero vector $v$ and matrix $A$, we see that matrices $A_1,A_2$, both with trace greater than $2$, are externally tangent if and only if $A_1-A_2^-$ has rank 1 and internally tangent if and only if $A_1-A_2$ has rank 1. This gives the following Observation:
\begin{observation}
Suppose the matrices $A_i,A_j,A_k$ are mutually externally tangent and have traces $>2$.  Then there are at most two matrices $B$ whose difference $A_s - B$ is rank 1 for each $s=i,j,k$: $B=A_m^-$ is a solution for any matrix $A_m$ externally tangent to $A_i,A_j,A_k$, and $B=A_m$ is a solution for any $A_m$ internally tangent to $A_i,A_j,A_k$.\qed
\label{o.succ}
\end{observation}
\noindent Note that the case of fewer than two solutions occurs when the triple of trace-2 circles of the down-set cones of the $A_i$ are tangent to a common line, leaving only one proper circle tangent to the triple.

\subsection{Viscosity Solutions}
\label{s.viscosity}

We would like to interpret the sandpile PDE $D^2 u \in \partial \Gamma$ in the classical sense, but the nonlinear structure of $\partial \Gamma$ makes this impractical.  Instead, we must adopt a suitable notion of weak solution, which for us is the {\em viscosity} solution. The theory of viscosity solutions is quite rich and we refer the interested reader to \cites{Crandall, Crandall-Ishii-Lions} for an introduction. Here we simply give the basic definitions.  We remark that these definitions and results make sense for any non-trivial subset $\Gamma \subseteq S_2$ that is downward closed and whose boundary has bounded trace (see Facts \ref{f.cl}, \ref{f.l2}, and \ref{f.g3} below).

If $\Omega \subseteq \R^2$ is an open set and $u \in C(\Omega)$, we say that $u$ satisfies the differential inclusion
\begin{equation}
\label{supersol}
D^2 u \in \bar \Gamma \quad \mbox{in } \Omega,
\end{equation}
if $D^2 \varphi(x) \in \bar \Gamma$ whenever $\varphi \in C^\infty(\Omega)$ touches $u$ from below at $x \in \Omega$.  Letting $\Gamma^c$ denote the closure of the complement of $\Gamma$, we say that $u$ satisfies
\begin{equation}
\label{subsol}
D^2 u \in \Gamma^c \quad \mbox{in } \Omega,
\end{equation}
if $D^2 \psi(x) \in \Gamma^c$ whenever $\psi \in C^\infty(\Omega)$ touches $u$ from above at $x \in \Omega$. Finally, we say that $u$ satisfies
\begin{equation*}
D^2 u \in \partial \Gamma \quad \mbox{in } \Omega,
\end{equation*}
if it satisfies both \eqref{supersol} and \eqref{subsol}.

The standard machinery for viscosity solutions gives existence, uniqueness, and stability of solutions. For example, the minimum in \eqref{obstacle} is indeed attained by some $v \in C(\R^2)$ and we have a comparison principle:

\begin{proposition}
\label{p.visc}
If $\Omega \subseteq \R^2$ is open and bounded and $u, v \in C(\bar \Omega)$ satisfy
\begin{equation*}
D^2 u \in \bar \Gamma \quad \mbox{and} \quad D^2 v \in \Gamma^c \quad \mbox{in } \Omega,
\end{equation*}
then $\sup_\Omega (v - u) = \sup_{\partial \Omega} (v - u)$.\qed
\end{proposition}

Recall that $C^{1,1}(U)$ is the class of differentiable functions on $U$ with Lipschitz derivatives.  In Section \ref{s.pq}, we construct piecewise quadratic $C^{1,1}$ functions which solve the sandpile PDE on each piece.  The following standard fact guarantees that the functions we construct are, in fact, viscosity solutions of the sandpile PDE on the whole domain (including at the interfaces of the pieces).

\begin{proposition}
\label{p.classical}
If $U\sbs\R^2$ is open, $u\in C^{1,1}(U)$, and for Lebesgue almost every $x\in U$
\[
D^2u(x)\mbox{ exists and } D^2u(x)\in \partial \Gamma,
\]
then $D^2u\in \partial \Gamma$ holds in the viscosity sense.\qed
\end{proposition}

\noindent Since we are unable to find a published proof, we include one here.

\begin{proof}
Suppose $\varphi \in C^\infty(U)$ touches $u$ from below at $x_0 \in U$. We must show $D^2 \varphi(x_0) \in \bar \Gamma$.  By approximation, we may assume that $\varphi$ is a quadratic polynomial. Fix a small $\ep > 0$. Let $A$ be the set of $y \in U$ for which there exists $p \in \R^2$ and $q \in \R$ such that
\begin{equation*}
\varphi_y(x) := \varphi(x) - \frac{1}{2} \ep |x|^2 + p \cdot x + q,
\end{equation*}
touches $u$ from below a $y$. Since $u \in C^{1,1}$, $p(y)$ is unique and that map $p : A \to \R^2$ is Lipschitz. Since $\ep > 0$ and $U$ is open, the image $p(A)$ contains a small ball $B_\delta(0)$. Thus we have
\begin{equation*}
0 < |B_\delta(0)| \leq |p(A)| \leq Lip(p) |A|. 
\end{equation*}
In particular, $A$ has positive Lebesgue measure and we may select a point $y \in A$ such that $D^2 u(y)$ exists and $D^2 u(y) \in \bar \Gamma$. Since $\varphi_h$ touches $u$ from below at $y$, we have $D^2 \varphi_y(y) \leq D^2 u(y)$ and thus $D^2 \varphi_y(y) = D^2 \varphi(y) - \ep I = D^2 \varphi(x_0) - \ep I \in \bar \Gamma$.  Sending $\ep \to 0$, we obtain $D^2 \varphi(x_0) \in \bar \Gamma$.
\end{proof}

\section{Algorithm to decide membership in \texorpdfstring{$\Gamma$}{Gamma}}
\label{s.gamma}

\noindent \emph{A priori}, the definition of $\Gamma$ does not give a method for verifying membership in the set.
In this section, we will show that matrices in $\Gamma$ correspond to certain infinite stabilizable sandpiles on $\Z^2$.  If $A \in \Gamma$ has rational entries, then its associated sandpile is periodic, which yields a method for checking membership in $\Gamma$ for any rational matrix, and allows us to algorithmically determine the height of the boundary of $\Gamma$ at any point with arbitrary precision.  Although restricting our attention in this section to the lattice $\Z^2$ simplifies notation a bit, we note that this algorithm generalizes past $\Z^2$, to allow the numerical computation of sets analogous to $\Gamma$ for other doubly periodic graphs in the plane, for which we have no exact characterizations (see Figure~\ref{f.gammatri}, for example).

If $q: \Z^2 \to \R$, write $\clg{q}$ for the function $\Z^2 \to \Z$ obtained by rounding each value of $q$ up to the nearest integer. The principal lemma is the following.

\begin{lemma}
\label{l.stab}
$A \in \Gamma$ if and only if the configuration $\Delta^1 \clg{q_A}$ is stabilizable, where
\begin{equation*}
q_A(x) := \frac{1}{2} x^t A x
\end{equation*}
is the quadratic form associated to $A$.
\end{lemma}

\begin{proof}
If $u$ satisfies \eqref{l.gamma}, then the Least Action Principle applied to $w = u - \clg{q_A}$ shows that $\eta = \Delta^1 \clg{q_A}$ is stabilizable.  On the other hand, if $\eta=\Delta^1\clg{q_A}$ is stabilizable with odometer $v$, then $u = v + \clg{q_A}$ satisfies \eqref{l.gamma}.
\end{proof}

Since $A\leq B$ implies $x^t A x \leq x^t Bx$ for all $x \in \Z^2$, the definition of $\Gamma$ implies that $\Gamma$ is downward closed in the matrix order:

\begin{fact}
\label{f.cl}
If $A\leq B$ and $B \in \Gamma$, then $A\in \Gamma$.\qed
\end{fact}
\noindent It follows that the boundary of $\Gamma$ is Lipschitz, and in particular, continuous; thus to determine the structure of $\Gamma$, it suffices to characterize the rational matrices in $\Gamma$.  We will say that a function $s$ on $\Z^2$ is \emph{$n$-periodic} if $s(x+y)=s(x)$ for all $y \in n\Z^2$.

\begin{lemma}
  If $A$ has entries in $\frac1n \Z$ for a positive integer $n$, then  $\Delta^1{\clg{q_A}}$ is $2n$-periodic.  
\label{l.per}
\end{lemma}
\begin{proof}
If $y \in 2n\Z^2$ then $Ay \in 2\Z^2$, so 
	\[ q_A(x+y) - q_A(x) = (x^t +  \frac12 y^t)Ay  \in \Z. \]  
Hence $\clg{q_A}-q_A$ is $2n$-periodic. Writing
\[
\Delta^1\clg{q_A}=\Delta^1 (\clg{q_A}-q_A) - \Delta^1 q_A
\]
and noting that $\Delta^1 q_A$ is constant, we conclude that $\Delta^1\clg{q_A}$ is $2n$-periodic.
\end{proof}

\noindent Thus the following lemma will allow us to make the crucial connection between rational matrices in $\Gamma$ stabilizable sandpiles on finite graphs.  It can be proved by appealing to \cite[Theorem~2.8]{Fey-Meester-Redig} on infinite toppling procedures, but we give a self-contained proof.

\begin{lemma}
\label{torus}
An $n$-periodic configuration $\eta : \Z^2 \to \Z$ is stabilizable if and only if it is stabilizable on the torus $T_n=\Z^2 / n \Z^2$.
\end{lemma}

\begin{proof}
Supposing $\eta$ is stabilizable on the torus $T_n$ with odometer $\bar v$, and extending $\bar v$ to an $n$-periodic function $v$ on $\Z^2$ in the natural way, we have that $\eta + \Delta^1 v \leq 3$. Thus $\eta$ is stabilizable on $\Z^2$ by the Least Action Principle.

Conversely, if $\eta$ is stabilizable on $\Z^2$, then there is a function $w : \Z^2 \to \N$ such that $\eta + \Delta^1 w \leq 3$. Proposition \ref{p.minD} implies that
\begin{equation*}
\tilde w(x) := \min \{ w(x + y) : y \in n \Z^2 \},
\end{equation*}
also satisfies $\eta + \Delta^1 \tilde w \leq 3$. Since $\tilde w$ is $n$-periodic, we also have $\eta + \Delta_{T_n}^1 \tilde w \leq 3$ and thus $\eta$ is stabilizable on the torus $T_n$.
\end{proof}

The preceding lemmas give us a simple prescription for checking whether a rational matrix $A$ is in $\Gamma$: compute $s=\Delta^1\clg{q_A}$ on the appropriate torus, and check if this is a stabilizable configuration.  To check that $s$ is stabilizable on the torus, we simply topple vertices with $\geq 4$ chips until either reaching a stable configuration, or until every vertex has toppled at least once, in which case Proposition \ref{p.v0} implies that $s$ is not stabilizable.

We thus can determine the boundary of $\Gamma$ to arbitrary precision algorithmically.  For $(a,b) \in \R^2$ let us define
	\[  c_0(a,b) = \sup\{ c \mid M(a,b,c) \in \Gamma \}. \]
By Fact~\ref{f.cl}, we have $M(a,b,c) \in \bar \Gamma$ if and only if $c \leq c_0(a,b)$.  Hence the boundary $\partial \Gamma$ is completely determined by the Lipschitz function $c_0(a,b)$.
In Figure \ref{f.numerical}, the shade of the pixel at $(a,b)$ corresponds to a value $c$ that is provably within $\frac{1}{1024}$ of $c_0(a,b)$. 

\medskip
\noindent
The above results are sufficient for confirmations for confirmation of properties of $\Gamma$ much more basic than the characterization from Theorem~\ref{t.apollo}. In particular, it is easy to deduce the following two facts:

\begin{fact}
If $A$ is rational and $\trace(A)<2$, then $A\in \Gamma$.\qed
\label{f.l2}
\end{fact}
\begin{fact}
If $A$ is rational and $\trace(A)>3$, then $A\not \in \Gamma$.\qed
\label{f.g3}
\end{fact}

In both cases, the relevant observation is that for rational $A$, $\trace(A)$ is exactly the average density of the corresponding configuration $\eta=\Delta^1\clg{q_A}$ on the appropriate torus.  This is all that is necessary for Fact \ref{f.g3}.  For Fact \ref{f.l2}, the additional observation needed (due to Rossin \cite{Rossin}) is that on any finite connected graph, a chip configuration with fewer chips than there are edges in the graph will necessarily stabilize: for unstabilizable configurations, a legal sequence toppling every vertex  at least once gives an injection from the edges of the graph to the chips, mapping each edge to the last chip to travel across it.

Facts~\ref{f.l2} and~\ref{f.g3} along with continuity imply that $2 \leq c_0(a,b) \leq 3$ for all $(a,b) \in \R^2$.  With additional work, but without requiring the techniques of \cite{2014integersuperharmonic}, the above results can be used to show that $c_0(a,b)=2$ for all $a \in 2\Z$ and $b\in \R$, confirming Theorem~\ref{t.apollo} along the vertical lines $x=a$ for $a \in 2\Z$. Finally, let us remark that $c_0$ has the translation symmetries 
	\[ c_0(a+2,b) = c_0(a,b) = c_0(a,b+2). \]
This follows easily from the observation that $\frac{1}{2}x(x+1)-\frac{1}{2} y(y+1)$ and $xy$ are both integer-valued discrete harmonic functions on $\Z^2$.

\section{Apollonian circle packings}
\label{s.packings}
For any three tangent circles $C_1,C_2,C_3$, we consider the corresponding triple of tangent closed discs $D_1,D_2,D_3$ with disjoint interiors.  We allow lines as circles, and allow the closure of any connected component of the complement of a circle as a closed disc.  Thus we allow internal tangencies, in which case one of the closed discs is actually the unbounded complement of an open bounded disc.  Note that to consider $C_1,C_2,C_3$ pairwise tangent we must require that three pairwise intersection points of the $C_i$ are actually distinct, or else the corresponding configuration of the $D_i$ is not possible.  In particular, there can be at most two lines among the $C_i$, which are considered to be tangent at infinity whenever they are parallel.

The three tangent closed discs $D_1,D_2,D_3$  divide the plane into exactly two regions; thus any pairwise triple of circles has two \emph{Soddy circles}, tangent to each circle in the triple.  If all tangencies are external and at most one of $C_1,C_2,C_3$ is a line, then exactly one of the two regions bordered by the $D_i$ is bounded, and the Soddy circle in the bounded region is called the \emph{successor} of the triple.

An \emph{Apollonian circle packing}, as defined in the introduction, is a minimal set of circles containing some triple of pairwise-tangent circles and closed under adding all Soddy circles of pairwise-tangent triples.   Similarly, a \emph{downward Apollonian circle packing} is a minimal set of circles containing some triple of pairwise externally tangent circles and closed under adding all successors of pairwise-tangent triples.

For us, the crucial example of an Apollonian packing is the Apollonian band packing.  This is the packing which appears in Theorem \ref{t.apollo}. A famous subset is the Ford circles, the set of circles $C_{p/q}$ with center $(\frac{2p}{q},\frac{1}{q^2})$ and radius $\frac{1}{q^2}$, where $p/q$ is a rational number in lowest terms.  A simple description of the other circles remains unknown, Theorem~\ref{t.apollo} provides an interesting new perspective.

\bigskip

  An important observation regarding Apollonian circle packings is that a triple of pairwise externally tangent circles is determined by its intersection points with its successor:
\begin{proposition}
\label{p.ydet}  
Given a circle $C$ and points $y_1,y_2,y_3\in C$, there is at exactly one choice of pairwise externally tangent circles $C_1,C_2,C_3$ which are externally tangent to $C$ at the points $y_1,y_2,y_3$.\qed
\end{proposition}

\noindent Proposition \ref{p.ydet}, together with its counterpart for the case allowing an internal tangency, allows the deduction of the following fundamental property of Apollonian circle packings.

\begin{proposition}
\label{p.mob}
Let $\ccc$ be an Apollonian circle packing.  A set $\ccc'$ of circles is an Apollonian circle packing if and only if $\ccc'=\mu(\ccc)$ for some M\"obius transformation $\mu$.\qed
\end{proposition}

\noindent The use of M\"obius transformations allows us to deduce a geometric rule based on medians of triangles concerning successor circles in Apollonian packings:

\begin{lemma}
  Suppose that circles $C,C_1,C_2$ are pairwise tangent, with Soddy circles $C_0$ and $C_3$, and let $z_i^2=p_i-c$, viewed as a complex number, where $c$ is the center of $C$ and $p_i$ is the intersection point of $C$ and $C_i$ for each $i$.  If $L_i$ is a line parallel to the vector $z_i$ which passes through 0 if $i=1,2,3$ and does not pass through 0 if $i=0$, then $L_3$ is a median line of the triangle formed by the lines $L_0,L_1,L_2$.
\label{L.median}
\end{lemma}
\begin{proof}
Without loss of generality, we assume that $C$ is a unit circle centered at the origin, and that $z_0^2=-1$.  The M\"obius transformation
  \begin{equation*}
    \mu_{z_1,z_2}(z)=\frac{z_1+z_1z_2-z(z_1-z_2)}{1+z_2+z(z_1-z_2)}
  \end{equation*}
sends $0$ to $z_1^2$, $1$ to $z_2^2$, and $\infty$ to $-1=z_0^2$.  Thus, for the pairwise tangent generalized circles $C'=\{y=0\}, C_0'=\{y=1\}, C_1'=\{x^2+(y-\frac 1 2)^2=\frac 1 4\}, C_2'=\{(x-1)^2+(y-\frac 1 2)^2=\frac 1 4\}, C_3'=\{(x-\frac 1 2)^2=\frac 1 {64}\}$ (these are some of the ``Ford circles''), we have that $\mu$ maps the intersection point of $C', C_i'$ to the intersection point of $C, C_i$ for $i=0,1,2$, thus it must map the intersection point of $C', C_3'$ to the intersection point of $C,C_3$, giving $\mu_{z_1,z_2}(\frac 1 2)=z_3$.  Thus it suffices to show that for
\begin{equation*}
  f(z_1,z_2):=\mu_{z_1,z_2}(1/2)=\frac{z_1+z_2+2z_1z_2}{z_1+z_2+2},
\end{equation*}
we have that
\begin{equation}
  f(z_1^2,z_2^2)=\frac{
\displaystyle\left(1+\frac{\Re {z_1}\Im {z_2}+\Re {z_2} \Im {z_1}}{2\Re {z_1}\Re {z_2}}i \right)^2
}
{
\displaystyle 1+\left(\frac{\Re {z_1}\Im {z_2}+\Re {z_2} \Im {z_1}}{2\Re {z_1}\Re {z_2}}\right)^2
},
\label{l.mediangoal}
\end{equation}
as the right-hand side is the square of the unit vector whose tangent is the average of the tangents of $z_1$ and $z_2$; this is the correct slope of our median line since $z_0^2=-1$ implies that $L_0$ is vertical.  We will check \eqref{l.mediangoal} by writing $z_1=\cos \alpha+i\sin \alpha$, $z_2=\cos \beta+i\sin \beta$ to rewrite $f(z_1^2,z_2^2)$ as 
\begin{multline}
\frac{
(\cos\alpha+i\sin\alpha)^2+(\cos\beta+i\sin\beta)^2+2(\cos\alpha+i\sin\alpha)^2(\cos\beta+i\sin\beta)^2
}
{
(\cos\alpha+i\sin\alpha)^2+(\cos\beta+i\sin\beta)^2+2
}\\=
\frac{
(\cos(\alpha+\beta)+i\sin(\alpha+\beta))(\cos(\alpha-\beta)+\cos(\alpha+\beta)+i\sin(\alpha+\beta))
}
{
\cos(\alpha-\beta)(\cos(\alpha+\beta)+i\sin(\alpha+\beta))+1
},
\label{l.fz2f}
\end{multline}
where we have used the identity 
\begin{equation*}
  (\cos x+i\sin x)^2+(\cos y+i\sin y)^2=2\cos(x-y)(\cos(x+y)+i \sin (x+y)),
\end{equation*}
which can be seen easily geometrically.
Dividing the top and bottom of the right side of \eqref{l.fz2f} by $\cos (\alpha+\beta)+i\sin(\alpha+\beta)$ gives
\begin{equation*}
f(z_1^2,z_2^2)=\frac{
\cos(\alpha-\beta)+\cos(\alpha+\beta)+i\sin(\alpha+\beta)
}
{
\cos(\alpha-\beta)+\cos(\alpha+\beta)-i\sin(\alpha+\beta)
}.
\end{equation*}
Thus to complete the proof, note that the right-hand side of \eqref{l.mediangoal} can be can simplified as
\begin{multline*}
  \frac{\left(1+\frac{\cos \alpha\sin \beta+\cos \beta \sin \alpha}{2\cos \alpha\cos \beta}i \right)^2}{1+\left(\frac{\cos \alpha\sin \beta+\cos \beta \sin \alpha}{2\cos \alpha\cos \beta}\right)^2}=
\frac{
\left(\cos(\alpha+\beta)+\cos(\alpha-\beta)+i\sin(\alpha+\beta)\right)^2
}
{
(\cos(\alpha+\beta)+\cos(\alpha-\beta))^2+\sin^2(\alpha+\beta)
}\\=
\frac{
\cos(\alpha+\beta)+\cos(\alpha-\beta)+i\sin(\alpha+\beta)
}
{
\cos(\alpha+\beta)+\cos(\alpha-\beta)-i\sin(\alpha+\beta)
}
\end{multline*}
by multiplying the top and bottom by $(2\cos\alpha\cos\beta)^2$ and using the Euler identity consequences
\begin{align*}
  2\cos\alpha\cos\beta&=\cos(\alpha+\beta)-\cos(\alpha-\beta)\\
  \cos\alpha\sin\beta&+\cos\beta\sin\alpha=\sin(\alpha+\beta).\qedhere
\end{align*}
\end{proof}

\begin{remark}
By Proposition \ref{p.mob}, a set of three points $\{x_1,x_2,x_3\}$ on a circle $C$ uniquely determine three other points $\{y_1,y_2,y_3\}$ on $C$, as the points of intersection of $C$ with successor circles of triples $\{C,C_i,C_j\}$, where $C_1,C_2,C_3$ are the unique triple of circles which are pairwise externally tangent and externally tangent to $C$ at the points $x_i$.  Since the median triangle of the median triangle of a triangle $T$ is homothetic to $T$, Lemma \ref{L.median} implies that this operation is an involution: the points determined by $\{y_1,y_2,y_3\}$ in this way is precisely the set $\{x_1,x_2,x_3\}$.
\label{r.dual}
\end{remark}

\begin{figure}[t]
\includegraphics[width=.5\linewidth]{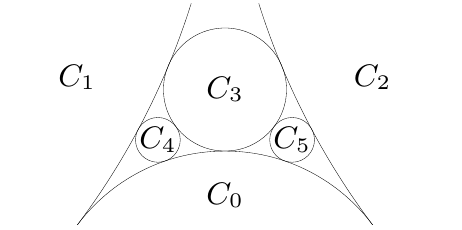}
\caption{\label{f.bounds} The circle arrangement from Proposition \ref{p.bounds}.}
\end{figure}

We close this section with a collection of simple geometric constraints on arrangements of externally tangent circles (Figure \ref{f.bounds}), whose proofs are rather straightforward:

\begin{proposition}
\label{p.bounds}
  Let $C_0,C_1,C_2$ be pairwise externally tangent proper circles with successor $C_3$, and let $C_4$ and $C_5$ be the successors of $C_0,C_1,C_3$ and $C_0,C_2,C_3$, respectively.  Letting $c_i$ denote the center of the circle $C_i$, we have the following geometric bounds:
  \begin{enumerate}
  \item $c_ic_3c_j\leq \pi$ for $\{i,j\}\sbs \{0,1,2\}.$ \label{i.csep}
  \item $\angle c_4 c_0 c_3,\angle c_5c_0c_3< \frac \pi 2.$ \label{i.maxsucc}
  \item $\angle c_4 c_0 c_3\geq \frac 1 2 \angle c_5 c_0 c_3$ (and vice versa).\label{i.closesucc}
  \item $\angle c_4c_3c_5\geq 2\cdot \arctan(3/4)$.\label{i.succsep}\qed
  \end{enumerate}
\end{proposition}

\section{Apollonian triangles and triangulations}
\label{s.geo}

\noindent We build up to Apollonian triangles and triangulations by defining the \emph{Apollonian curve} associated to an ordered triple of circles.  This will allow us to define the \emph{Apollonian triangle} associated to a quadruple of circles, and finally the \emph{Apollonian triangulation} associated to a downward packing of circles.  We will define these objects implicitly, and then show that they exist and are unique up to translation and homothety (i.e., any two Apollonian curves $\gamma, \gamma'$ associated to the same triple satisfy $\gamma' = a\gamma + \mathbf{b}$ for some $a \in \R$ and $b \in \R^2$).   In Section \ref{s.arec}, we give a recursive description of the Apollonian curves which characterizes these objects without reference to circle packings.

Fix a circle $C_0$ with center $c_0$ and let $C$ and $C'$ be tangent circles tangent to $C_0$ at $x$ and $x'$, and have centers $c$ and $c'$, respectively.  We define $s(C,C')$ to be the successor of the triple $(C_0,C,C')$ and $\alpha(C)$ to be the angle of the vector $v(C):= c-c_0$ with the positive $x$-axis.  Let $v^{1/2}(C)$ to be a complex square root of $v(C)$, and let $\ell^{1/2}(C) = \R v^{1/2}(C)$ be the real line it spans.  (We will actually only use $\ell^{1/2}(C)$, so the choice of square root is immaterial.)  Note that all of these functions depend on the circle $C_0$; we will specify which circle the functions are defined with respect to when it is not clear from context.  

Now fix circles $C_1$ and $C_2$ such that $C_0,C_1,C_2$ are pairwise externally tangent.  Let $\ccc$ denote the smallest set of circles such that $C_1,C_2 \in \ccc$ and for all tangent $C,C' \in \ccc$ we have $s(C,C') \in \ccc$.  
Note that all circles in $\ccc$ are tangent to $C_0$.  

\begin{definition}
\label{d.appcurve}  A (continuous) curve $\gamma:[\alpha(C_1),\alpha(C_2)]\to \R^2$ is an \emph{Apollonian curve} associated to the triple $(C_0,C_1,C_2)$ if for all tangent circles $C,C'\in \ccc$,	\[ \gamma(\alpha(C))-\gamma(\alpha(C')) \in \ell^{1/2}(s(C,C')). \]
\end{definition}
We call $\gamma(\alpha(s(C_1,C_2)))$ the \emph{splitting point} of $\gamma$.  The following Observation implies, in particular, that the splitting point divides $\gamma$ into two smaller Apollonian curves.
\begin{observation}
For any two tangent circles $C,C'\in \ccc$, the restriction $\gamma|_{[\alpha(C),\alpha(C')]}$ is also an Apollonian curve.\qed
\label{o.subcurve}
\end{observation}

To prove the existence and uniqueness of Apollonian curves, we will need the following observation, which is easy to verify from the fact that no circle lying inside the region bounded by $C_0,C_1,C_2$ and tangent to $C_0$ has interior disjoint from the family $\ccc$:
\begin{observation}
\label{o.dense}
$\alpha(\ccc)$ is dense in the interval $[\alpha(C_1),\alpha(C_2)]$. \qed
\end{observation}

\noindent We can now prove the existence and uniqueness of Apollonian curves.

\begin{theorem}
\label{t.appcurve}
  For any pairwise 
  tangent ordered triple of circles $(C_0,C_1,C_2)$, there is an associated Apollonian curve $\gamma$, which is unique up to translation and scaling.
\end{theorem}

\begin{proof}
  The choice of the points $\gamma(\alpha(C_1))$ and $\gamma(\alpha(C_2))$ is determined uniquely up to translation and scaling by the constraint that $\gamma(\alpha(C_1))-\gamma(\alpha(C_2))$ is a real multiple of $v^{1/2}(s(C_1,C_2))$.  This choice then determines the image $\gamma(\alpha(C))$ for all circles $C\in \ccc$ recursively: for any tangent circles $C^1,C^2 \in  \ccc$ with $C^3:=s(C^1,C^2)$ the constraints 
\begin{align*}
\gamma(\alpha(C^1))-\gamma(\alpha(C^3)) \in \ell^{1/2}(s(C^1,C^3)) \\
\gamma(\alpha(C^2))-\gamma(\alpha(C^3)) \in \ell^{1/2}(s(C^2,C^3))
\end{align*}
determine $\gamma(\alpha(C^3))$ uniquely given $\gamma(\alpha(C^1))$ and $\gamma(\alpha(C^2))$.  To show that there is a unique and well-defined curve $\gamma$, by Observation~\ref{o.dense} it is enough to show that $\gamma$ is a continuous function on the set $\alpha(\ccc)$.  For this it suffices to find an absolute constant $\beta<1$ for which
\begin{equation}
\label{l.dbound}
      \abs{\gamma(\alpha(C^1))-\gamma(\alpha(s(C^1,C^2)))}\leq
\beta \abs{\gamma(\alpha(C^1))-\gamma(\alpha(C^2))}
\end{equation}
for tangent circles $C^1,C^2\in \ccc$, as this implies, for example, that by taking successors $k$ times, we can find a circle $C'\in \ccc$ such that all points in $\gamma([\alpha(C^1),\alpha(C')])$ lie within $\beta^k\abs{\gamma(\alpha(C^1))-\gamma(\alpha(C^2))}$ of $\gamma(\alpha(C^1))$.  We get the absolute constant $\beta$ from an application of the law of sines to the triangle with vertices $p_1=\gamma(\alpha(C^1)),p_2=\gamma(\alpha(C^2)),p_3=\gamma(\alpha(s(C^1,C^2)))$:  part  \ref{i.closesucc} of Proposition \ref{p.bounds} implies that $\theta:=\angle p_3p_1p_2\geq \frac 1 2 \angle p_3p_2p_1$; the Law of Sines then implies that line \eqref{l.dbound} holds with $\beta=\frac{\sin(2\theta)}{\sin(3\theta)}$, which is $\leq \frac 2 3$ always since part \ref{i.maxsucc} of Proposition \ref{p.bounds} implies that $\theta \leq \frac \pi 2$.
\end{proof}

\begin{theorem}
  The image of an Apollonian curve $\gamma$ corresponding to $(C_0,C_1,C_2)$ 
    has a unique tangent line at each point $\gamma(\alpha)$. This line is at angle $\alpha/2$ to the positive $x$-axis.  In particular, $\gamma$ is a convex curve.
\label{t.appconv}
\end{theorem}
\begin{proof}
Observation~\ref{o.dense} and Definition \ref{d.appcurve} give that for any $C\in \ccc$, there is a unique line tangent to the image of $\gamma$ at $\gamma(\alpha(C))$, which is at angle $\alpha(C)/2$ to the $x$-axis.  Together with another application of Observation~\ref{o.dense} and the fact that $\frac{\alpha}{2}$ is a continuous function of $\alpha$, this gives that the image $\gamma$ has a unique tangent line at angle $\frac{\alpha}{2}$ to the $x$-axis at any point $\gamma(\alpha)$. \end{proof}

\begin{definition}
  The \emph{Apollonian triangle} corresponding to an unordered triple of externally tangent circles $C_1,$ $C_2,$ $C_3$ and circle $C_0$ externally tangent to each of them is defined as the bounded region (unique up to translation and scaling) enclosed by the images of the Apollonian curves $\gamma_{12}$, $\gamma_{23}$, $\gamma_{31}$ corresponding to the triples $(C_0,C_1,C_2)$, $(C_0,C_2,C_3)$, $(C_0,C_3,C_1)$ such that $\gamma_{ij}(\alpha(C_j))=\gamma_{jk}(\alpha(C_j))$ for each $\{i,j,k\}=\{1,2,3\}$.
\label{d.tri}
\end{definition}

Note that Theorem \ref{t.appcurve} implies that each triple $\{C_1,C_2,C_3\}$ of pairwise tangent circles corresponds to an Apollonian triangle $\ttt$ which is unique up to translation and scaling.  Theorem \ref{t.appconv} implies that the curves $\gamma_{12},\gamma_{23},\gamma_{31}$ do not intersect except at their endpoints,
and that $\ttt$ is strictly contained in the triangle with vertices $\gamma_{12}(C_2), \gamma_{23}(C_3), \gamma_{31}(C_1)$.  Another consequence of Theorem \ref{t.appconv} is that any two sides of an Apollonian triangle have the same tangent line at their common vertex.  Thus, the interior angles of an Apollonian triangle are $0$.

An Apollonian triangle is \emph{proper} if $C_4$ is smaller than each of $C_1,C_2,C_3$, i.e., if $C_4$ is the successor of $C_1,C_2,C_3$, and all Apollonian triangles appearing in our solutions to the sandpile PDE will be proper.

We also define a degenerate version of an Apollonian triangle:
\begin{definition}
  The \emph{degenerate Apollonian triangle} corresponding to the pairwise tangent circles $(C_1,C_2,C_3)$ is the compact region (unique up to translation and scaling) enclosed by the image of the Apollonian curve $\gamma$ corresponding to $(C_1,C_2,C_3)$, and the tangent lines to $\gamma$ at its endpoints $\gamma(\alpha(C_2))$ and $\gamma(\alpha(C_3))$.
\end{definition}

Proper Apollonian triangles (and their degenerate versions) are the building blocks of Apollonian triangulations, the fractals that support piecewise-quadratic solutions to the sandpile PDE.  Recall that $\aaa^-(C_1,C_2,C_3)$ denotes the smallest set of circles containing the circles $C_1,C_2,C_3$ and closed under adding successors of pairwise tangent triples.  To each circle $C \in \aaa^-(C_1,C_2,C_3) \setminus \{C_1,C_2,C_3\}$ we associate an Apollonian triangle $\ttt_C$ corresponding to the unique triple $\{C^1,C^2,C^3\}$ in  $\aaa^-(C_1,C_2,C_3)$ whose successor is $C$.

\begin{definition}
  The \emph{Apollonian triangulation} associated to a triple $\{C_1,C_2,C_3\}$ of externally tangent circles is a union of (proper) Apollonian triangles $\ttt_C$ corresponding to each circle $C\in \aaa^-(C_1,C_2,C_3)\stm \{C_1,C_2,C_3\}$, together with degenerate Apollonian triangles $\ttt_C$ for each $C=C_1,C_2,C_3$, such that disjoint circles correspond to disjoint Apollonian triangles, and such that for tangent circles $C,C'$ in $\aaa^-(C_1,C_2,C_3)$ where $r(C')\leq r(C)$, we have that $\ttt_{C'}$ and $\ttt_{C}$ intersect at a vertex of $\ttt_{C'}$, and that their boundary curves meet at right angles.
\end{definition}
\noindent Figure \ref{f.sandfractal} shows an Apollonian triangulation, excluding the three degenerate Apollonian triangles on the outside.

\begin{remark}
\label{r.meetingpoint}
By Theorem \ref{t.appconv} and the fact that centers of tangent circles are separated by an angle $\pi$ about their tangency point, the right angle requirement is equivalent to requiring that the intersection of $\ttt_{C'}$ and $\ttt_C$ occurs at the point $\gamma(\alpha(C'))$ on an Apollonian boundary curve $\gamma$ of $\ttt_C$.
\end{remark}

\section{Geometry of Apollonian curves}
\label{s.arec}
In this section, we will give a circle-free geometric description of Apollonian curves.  This will allow us to easily deduce geometric bounds necessary for our construction of piecewise-quadratic solutions to work.  

Recall that by Theorem \ref{t.appconv}, each pair of boundary curves of an Apollonian triangle have a common tangent line where they meet.  Denoting the three such tangents the \emph{spline lines} of the Apollonian triangle, Remark \ref{r.dual}, and Lemma \ref{L.median} give us the following:
\begin{lemma}
The spline lines of an Apollonian triangle with vertices $v_1,v_2,v_3$ are the median lines of the triangle $\tri v_1v_2v_3$, and thus meet at a common point, which is the centroid of $\tri v_1v_2v_3$\qed.
\label{l.spline}
\end{lemma}

More crucially, Lemma \ref{L.median} allows us to give a circle-free description of Apollonian curves.  Indeed, letting $c$ be the intersection point of the tangent lines to the endpoints $p_1,p_2$ of an Apollonian curve $\gamma$, Lemma \ref{L.median} implies (via Definition \ref{d.appcurve} and Theorem \ref{t.appconv}) that the splitting point $s$ of $\gamma$ is the intersection of the medians from $p_1,p_2$ of the triangle $\tri p_1p_2c$, and thus the centroid of the triangle $\tri p_1p_2c$.  The tangent line to $\gamma$ at $s$ is parallel $p_1p_2$; thus, by Observations \ref{o.subcurve} and \ref{o.dense}, the following recursive procedure determines a dense set of points on the curve $\gamma$ given the triple $(p_1,p_2,c)$:
\begin{enumerate}
\item find the splitting point $s$ as the centroid of $\tri p_1p_2 c$.
\item compute the intersections $c_1,c_2$ of the $p_1c$ and $p_2c$, respectively, with the line through $s$ parallel to $p_1p_2$.
\item carry out this procedure on the triples $(p_1,c_1,s)$ and $(s,c_2,p_2)$.
\end{enumerate}

By recalling that the centroid of a triangle lies 2/3 of the way along each median, the correctness of this procedure thus implies that the ``generalized quadratic B\'ezier curves'' with constant $\frac 1 3$ described by Paoletti in his thesis \cite{pathesis} are Apollonian curves.  Combined with Lemma \ref{l.spline}, this procedure also gives a way of enumerating barycentric coordinates for a dense set of points on each of the boundary curves of an Apollonian triangle, in terms of its 3 vertices.  Thus, in particular, all Apollonian triangles are equivalent under affine transformations. Conversely, since Proposition \ref{p.ydet} implies that any 3 vertices in general position have a corresponding Apollonian triangle, the affine image of any Apollonian triangle must also be an Apollonian triangle.  In particular:

\begin{theorem}
  For any three vertices $v_1,v_2,v_3$ in general position, there is a unique Apollonian triangle whose vertices are $v_1,v_2,v_3$.\qed
\label{t.Auni}
\end{theorem}

Another consequence of the affine equivalence of Apollonian triangles is conformal \emph{inequivalence} of Apollonian triangulations: suppose $\phi : \sss \to \sss'$ is a conformal map between Apollonian triangulations which preserves the incidence structure. Let $\ttt$ and $\ttt'$ be their central Apollonian triangles, and $\alpha : \ttt \to \ttt'$ the corresponding affine map.  By Remark \ref{r.meetingpoint}, the points  on $\partial \ttt$ computed by the recursive procedure above are the points at which $\ttt$ is incident to other Apollonian triangles of $\sss$; thus, $\phi = \alpha$ on a dense subset of $\partial \ttt,$ and therefore on all of $\partial \ttt$. Since the real and imaginary parts of $\phi$ and $\alpha$ are harmonic, the maximum principle implies that $\phi = \alpha$ on $\ttt,$ and therefore on $\sss$ as well, giving that $\sss$ and $\sss'$ are equivalent under a Euclidean similarity transformation.   We stress that in general, even though $\ttt$ and $\ttt'$ are affinely equivalent, nonsimilar triangulations are \emph{not} affinely equivalent, as can be easily be verified by hand.

It is now easy to see from the right-angle requirement for Apollonian triangulations that the Apollonian triangulation associated to a particular triple of circles must be also be unique up to translation and scaling: by Remark~\ref{r.meetingpoint}, the initial choice of translation and scaling of the three degenerate Apollonian triangles determines the rest of the figure.  (On the other hand, it is not at all obvious that Apollonian triangulations exist. This is proved in Theorem~\ref{t.sandfrac} below.)  Hence by Proposition \ref{p.ydet}, an Apollonian triangulation is uniquely determined by the three pairwise intersection points of its three degenerate triangles:
\begin{theorem}
  For any three vertices $v_1,v_2,v_3$, there is at most one Apollonian triangulation for which the set of vertices of its three degenerate Apollonian triangles is $\{v_1,v_2,v_3\}$.\qed
\end{theorem}

To ensure that our piecewise-quadratic constructions are well-defined on a convex set, we will need to know something about the area of Apollonian triangles.  Affine equivalence implies that there is a constant $C$ such that the area of any Apollonian triangle is equal to $C\cdot A(T)$ where $T$ is the Euclidean triangle with the same 3 vertices.  In fact we can determine this constant exactly:
\begin{lemma}
  An Apollonian triangle $\ttt$ with vertices $p_1,p_2,p_3$ has area $\frac 4 7 A(T)$ where $A(T)$ is the area of the triangle $T=\tri p_1p_2p_3$.
\label{L.47}
\end{lemma}
\begin{proof}
Lemma \ref{l.spline} implies that the spline lines of $\ttt$ meet at the centroid $c$ of $T$.
It suffices to show that $A(\ttt\cap \tri p_ip_jc)=\frac 4 7A(\tri p_ip_jc)$ for each $\{i,j\}\sbs \{1,2,3\}$; thus, without loss of generality, we will show that this holds for $i=1,j=2$.  

Let $\ttt_3=\ttt\cap \tri p_1p_2c$, and let $\ttt_3^C=\tri p_1p_2c\stm \ttt_3$.  We aim to compute the area of the complement $\ttt_3^C$ using our recursive description of Apollonian curves.  Step 1 of each stage of the recursive description computes a splitting point $s'$ relative to points $p_1',p_2',c'$, and $\ttt_3^C$ is the union of the triangles $\tri p_1'p_2's'$ for all such triples of points encountered in the procedure.  As the median lines of any triangle divide it into 6 regions of equal area, we have for each such triple that $A(p_1'p_2's')=\frac 1 3A(p_1'p_2'c').$

Meanwhile, step 2 each each stage of the recursive construction computes new intersection points $c_1',c_2'$ with which to carry out the procedure recursively.  The sum of the area of the two triangles $\tri p_1',c_1',s'$ and $\tri s'c_2' p_2'$ is 
\[
A(\tri p_1',c_1',s')+A(\tri s',c_2',p_2')=\tfrac 5 9 A(\tri p_1'p_2' s')-\tfrac 1 3 A(\tri p_1'p_2' s')=\tfrac 2 9 A(\tri p_1'p_2' s'),
\]
Since $\frac 5 9 A(\tri p_1'p_2' s')$ is the portion of the area of the triangle $p_1'p_2's'$ which lies between the lines $p_1p_2$ and $c_1',c_2'$.  Thus, the area $A(\ttt_3^C)$ is given by 
\begin{equation*}
A(\tri p_1p_2c)\cdot \left(\tfrac 1 3+\left(\tfrac 2 9\right) \tfrac 1 3 + \left(\tfrac 2 9\right)^2 \tfrac 1 3 + \left(\tfrac 2 9\right)^3 \tfrac 1 3+\cdots\right)=
\tfrac 3 7 A(\tri p_1p_2 c).\qedhere
\end{equation*}
\end{proof}

We conclude this section with some geometric bounds on Apollonian triangles.
The following Observation is easily deduced from part \ref{i.succsep} of Proposition \ref{p.bounds}:
\begin{observation}
  Given a proper Apollonian triangle with vertices $v_1,v_2,v_3$ generated from a non-initial circle $C$ and parent triple of circles $(C_1,C_2,C_3)$, the angles $\angle v_iv_jv_k$ $(\{i,j,k\}=\{1,2,3\})$ are all $> \arctan(3/4)>\frac \pi 5$ if $C$ has smaller radius than each of $C_1,C_2,C_3$.\qed
\label{o.Aintangles}
\end{observation}

Recall that Theorem \ref{t.appconv} implies that pairs of boundary curves of an Apollonian triangle meet their common vertex at a common angle, and that there is thus a unique line tangent to both curves through their common vertex.  We call such lines $L_1,L_2,L_3$ for each vertex $v_1,v_2,v_3$ the \emph{median lines} of the Apollonian triangle, motivated by the fact that Lemma \ref{L.median} implies that they are median lines of the triangle $\tri v_1v_2v_3$.

\begin{observation}
\label{o.Acntangles}
The pairwise interior angles of the median lines $L_1,L_2,L_3$ of a proper Apollonian triangle all lie in the interval $(\frac \pi 2,\frac {3\pi}{4})$.
\end{observation}

\begin{proof}
Part \ref{i.csep} of Proposition \ref{p.bounds}  gives that the interior angles of the median lines of the corresponding Apollonian triangle must satisfy $\alpha_i\leq \pi-\frac \pi 4=\frac 3 4\pi$.  The lower bound follows from $\alpha_1+\alpha_2+\alpha_3=2\pi$.
\end{proof}

\section{Fractal solutions to the sandpile PDE}
\label{s.pq}

\noindent Our goal now is to prove that Apollonian triangulations exist, and that they support piecewise quadratic solutions to the sandpile PDE which have constant Hessian on each Apollonian triangle. We prove the following theorems in this section:

\begin{theorem}
To any mutually externally tangent circles $C_1,C_2,C_3$ in an Apollonian circle packing $\aaa$, there exists a corresponding Apollonian triangulation $\sss$.  Moreover, the closure of $\sss$ is convex.
\label{t.sandfrac}
\end{theorem}

\begin{theorem}
  For any Apollonian triangulation $\sss$ there is a piecewise quadratic $C^{1,1}$ map $u: \bar \sss \to \R$ such that for each Apollonian triangle $\ttt_C$ comprising $\sss$, the Hessian $D^2 u$ is constant and equal to $m(C)$ in the interior of $\ttt_C$.
\label{t.pq}
\end{theorem}

Theorem \ref{t.pq} implies Theorem \ref{t.mainpq} from the Introduction via Proposition \ref{p.classical}, by taking $Y=\sss$ and $Z=\bar \sss$, where $\sss=\sss(A_1,A_2,A_3)$ is the Apollonian triangulation generated by the triple of circles $c(A_i)$ for $i=1,2,3$.  Using the fact that $\sss$ has full measure in $\bar \sss$, proved in Section \ref{s.fm}, this theorem constructs piecewise-quadratic solutions to the sandpile PDE via Proposition \ref{p.classical}. 

We will prove Theorems \ref{t.sandfrac} and \ref{t.pq} in tandem; perhaps surprisingly, we do not see a simple geometric proof of Theorem \ref{t.sandfrac}, and instead, in the course of proving Theorem \ref{t.pq}, will prove that certain piecewise-quadratic approximations to $u$ exist and use constraints on such constructions to achieve a recursive construction of approximations to $\sss$.

\subsection{The recursive construction}
We begin our construction of $u$---and, simultaneously $\sss$, which will be the limit set of the support of the approximations to $u$ we construct---by considering the three initial matrices $A_i = m(C_i)$ for $i=1,2,3$.

Observation \ref{o.succ} implies that there are vectors $v_1,v_2,v_3$ such that 
\begin{equation*}
A_i = A_4^- + v_i \otimes v_i \quad \mbox{for each } i=1,2,3.
\end{equation*}
We may then select distinct $p_1, p_2, p_3 \in \R^2$ such that $v_i \cdot (p_j - p_k) = 0$ for $\{ i, j, k \} = \{ 1, 2, 3 \}$.  Observation \ref{o.vecsq} and Definition \ref{d.appcurve} imply that we can choose degenerate Apollonian triangles $\ttt_{A_i}$ corresponding to $(A_i,A_j,A_k)$ ($\{i,j,k\}=\{1,2,3\}$) meeting at the points $p_1,p_2,p_3$.  Note that the straight sides of distinct $\ttt_{A_i}$ meet only at right angles.

It is easy to build a piecewise quadratic map $u_0 \in C^{1,1}(\ttt_{A_1}\cup \ttt_{A_2}\cup \ttt_{A_3})$ whose Hessian lies in the set $\{ A_1, A_2, A_3 \}$:  for example, we can simply define $u_0$ as
\begin{equation}
u_0(x) := \frac{1}{2} x^t A_4^- x + \frac{1}{2} (v_i \cdot (x - p_j))^2 \quad \mbox{for }x \in \ttt_{A_i}\mbox{ and }i \neq j.
\label{l.u0}
\end{equation}

We now extend this map to the full Apollonian triangulation by recursively choosing quadratic maps on successor Apollonian triangles that are compatible with the previous pieces. The result is a piecewise-quadratic $C^{1,1}$ map whose pieces form a full measure subset of a compact set.
By a \emph{quadratic function} on $\R^2$ we will mean a function of the form $\phi(x) = x^t A x + b^t \cdot x + c$ for some matrix $A \in S_2$, vector $b \in \R^2$ and $c \in \R$.
Letting $(1,2,3)^3$ denote $\{(1,2,3),(2,3,1),(3,1,2)\}$, the heart of the recursion is the following claim, illustrated in Figure \ref{f.claim}.

\begin{figure}[t]
\begin{tikzpicture}[scale=1.4,thick]
\filldraw (-0.00000000,-1.73205081) circle (1pt);
\filldraw (1.73205081,0.86602540) circle (1pt);
\filldraw (-1.73205081,0.86602540) circle (1pt);
\filldraw (0.00000000,1.15470054) circle (1pt);
\filldraw (-0.86602540,-0.57735027) circle (1pt);
\filldraw (0.86602540,-0.57735027) circle (1pt);
\draw (-0.00000000,-1.93205081) node {$p_1$};
\draw (1.93205081,0.86602540) node {$p_2$};
\draw (-1.93205081,0.86602540) node {$p_3$};
\draw (0.00000000,1.35470054) node {$q_1$};
\draw (-1.06602540,-0.67735027) node {$q_2$};
\draw (1.06602540,-0.67735027) node {$q_3$};
\draw (-1.73205081,0.86602540) -- (-1.60375075,0.90879209) -- (-1.41130066,0.96225045) -- (-1.17608388,1.01570881) -- (-0.96225045,1.05847549) -- (-0.72703367,1.09055051) -- (-0.44905021,1.12262552) -- (-0.19245009,1.14400887) -- (0.00000000,1.15470054) -- (0.19245009,1.14400887) -- (0.44905021,1.12262552) -- (0.72703367,1.09055051) -- (0.96225045,1.05847549) -- (1.17608388,1.01570881) -- (1.41130066,0.96225045) -- (1.60375075,0.90879209) -- (1.73205081,0.86602540);
\draw (-0.00000000,-1.73205081) -- (-0.06415003,-1.65720911) -- (-0.16037507,-1.53960072) -- (-0.27798346,-1.38991731) -- (-0.38490018,-1.25092558) -- (-0.50250857,-1.09055051) -- (-0.64150030,-0.89810042) -- (-0.76980036,-0.71634200) -- (-0.86602540,-0.57735027) -- (-0.96225045,-0.42766687) -- (-1.09055051,-0.22452510) -- (-1.22954224,0.00000000) -- (-1.34715063,0.19245009) -- (-1.45406734,0.37420851) -- (-1.57167573,0.57735027) -- (-1.66790078,0.74841702) -- (-1.73205081,0.86602540);
\draw (1.73205081,0.86602540) -- (1.66790078,0.74841702) -- (1.57167573,0.57735027) -- (1.45406734,0.37420851) -- (1.34715063,0.19245009) -- (1.22954224,-0.00000000) -- (1.09055051,-0.22452510) -- (0.96225045,-0.42766687) -- (0.86602540,-0.57735027) -- (0.76980036,-0.71634200) -- (0.64150030,-0.89810042) -- (0.50250857,-1.09055051) -- (0.38490018,-1.25092558) -- (0.27798346,-1.38991731) -- (0.16037507,-1.53960072) -- (0.06415003,-1.65720911) -- (-0.00000000,-1.73205081);
\draw (0.00000000,1.15470054) -- (-0.01069167,1.07629495) -- (-0.03207501,0.96225045) -- (-0.06415003,0.82682261) -- (-0.09622504,0.70565033) -- (-0.13899173,0.57735027) -- (-0.19245009,0.42766687) -- (-0.24590845,0.29223903) -- (-0.28867513,0.19245009) -- (-0.34213349,0.09978894) -- (-0.41697519,-0.02138334) -- (-0.50250857,-0.14968340) -- (-0.57735027,-0.25660012) -- (-0.65219197,-0.34926127) -- (-0.73772534,-0.44905021) -- (-0.81256705,-0.52745580) -- (-0.86602540,-0.57735027);
\draw (-0.86602540,-0.57735027) -- (-0.80187537,-0.54883914) -- (-0.70565033,-0.51320024) -- (-0.58804194,-0.47756133) -- (-0.48112522,-0.44905021) -- (-0.36351684,-0.42766687) -- (-0.22452510,-0.40628352) -- (-0.09622504,-0.39202796) -- (0.00000000,-0.38490018) -- (0.09622504,-0.39202796) -- (0.22452510,-0.40628352) -- (0.36351684,-0.42766687) -- (0.48112522,-0.44905021) -- (0.58804194,-0.47756133) -- (0.70565033,-0.51320024) -- (0.80187537,-0.54883914) -- (0.86602540,-0.57735027);
\draw (0.86602540,-0.57735027) -- (0.81256705,-0.52745580) -- (0.73772534,-0.44905021) -- (0.65219197,-0.34926127) -- (0.57735027,-0.25660012) -- (0.50250857,-0.14968340) -- (0.41697519,-0.02138334) -- (0.34213349,0.09978894) -- (0.28867513,0.19245009) -- (0.24590845,0.29223903) -- (0.19245009,0.42766687) -- (0.13899173,0.57735027) -- (0.09622504,0.70565033) -- (0.06415003,0.82682261) -- (0.03207501,0.96225045) -- (0.01069167,1.07629495) -- (0.00000000,1.15470054);
\end{tikzpicture}
\caption{The Apollonian curves $\gamma_i,\gamma'_i$ $(i=1,2,3)$ in the claim.\label{f.claim}}
\end{figure}
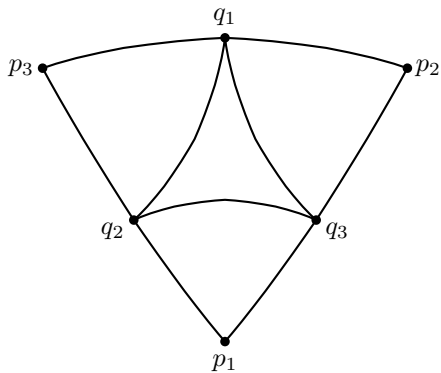

\begin{claim*}
Suppose $B_0$ is the successor of a triple $(B_1,B_2,B_3)$, and that for each $(i,j,k)\in (1,2,3)^3$, we have that $\gamma_i$ is an Apollonian curve for $(B_i,B_j,B_k)$ from $p_k$ to $p_j$, $\phi_i$ is a quadratic function with Hessian $B_i$, and the value and gradient of $\phi_i,\phi_j$ agree at $p_k$ for each $k$.  

Then there is a quadratic function $\phi_0$ with Hessian $B_0$ whose value and gradient agree with that of $\phi_i$ at each $q_i:=\gamma_i(\alpha_i(B_0))$, and for each $(i,j,k)\in(1,2,3)^3$, there is an Apollonian curve $\gamma_i'$ from $q_j$ to $q_k$ corresponding to the triple $(B_0,B_j,B_k)$.  (Here, the $\alpha_i$ denotes the angle function $\alpha$ defined with respect to $B_i$.) 
\end{claim*}

We will first see how the claim allows the construction to work.  Defining the \emph{level} of each $A_1,A_2,A_3$ to be $\ell(A_i)=0$, and recursively setting the level of a successor of a triple $(A_i,A_j,A_k)$ as $\max(\ell(A_i),\ell(A_j),\ell(A_k))+1$, allows us to define a \emph{level-$k$ partial Apollonian triangulation} which will be the domain of our iterative constructions.

\begin{definition}
  A \emph{level-$k$ partial Apollonian triangulation} corresponding to $\{A_1,A_2,A_3\}$ is the subset $\sss_k\sbs \sss(A_1,A_2,A_3)$ consisting of the union of the Apollonian triangles $\ttt_{A}\in \sss$ for which $\ell(A)\leq k$.
\end{definition}
\noindent Note that $u_0$ is defined on a level-0 partial Apollonian triangulation.

Consider now a $C^{1,1}$ piecewise-quadratic function $u_{k-1}$ defined on the union of a level-$(k-1)$ partial Apollonian triangulation $\sss_{k-1}$, whose Hessian on each $\ttt_{A_i}\in \sss_{k-1}$ is the matrix $A_i$.  Any three pairwise intersecting triangles $\ttt_{A_i},\ttt_{A_j},\ttt_{A_k}\in \sss_{k-1}$ bound some region $R$, and, denoting by $\gamma_s$ the boundary curve of each $\ttt_{A_s}$ which coincides with the boundary of $R$ and by $p_s$ the shared endpoint of $\gamma_t,\gamma_u$ $(\{s,t,u\}=\{i,j,k\})$, the hypotheses of the Claim are satisfied for $(B_1,B_2,B_3)=(A_i,A_j,A_k)$, where $\phi_1,\phi_2,\phi_3$ are the quadratic extensions to the whole plane of the restrictions $u_{k-1}|_{\ttt_{A_i}},u_{k-1}|_{\ttt_{A_j}},u_{k-1}|_{\ttt_{A_k}}$, respectively.

Noting that the three Apollonian curves given by the claim bound an Apollonian triangle corresponding to the triple $(A_i,A_j,A_k)$, the claim allows us to extend $u_{k-1}$ to a $C^{1,1}$ function $u_{k}$ on the level-$k$ partial fractal $\sss_k$ by setting $u_{k}=\phi_0$ on the triangle $\ttt_{A_\ell}\in \sss_k$ for the successor $A_{\ell}$ of $(A_i,A_j,A_k)$, for each externally tangent triple $\{A_i,A_j,A_k\}$ in $\sss_{k-1}$.  Letting $U$ denote the topological closure of $\sss$, we can extend the limit $\bar u:\sss\to \R$ of the $u_k$ to a $C^{1,1}$ function $u:U\to \R$; to prove Theorems \ref{t.sandfrac} and \ref{t.pq}, it remains to prove the Claim, and that $\sss$ is a full-measure subset of its convex closure $Z$, so that in fact $U=Z$.  We will prove that $\sss$ is full-measure in $Z$ in Section \ref{s.fm}, and so turn our attention to proving the Claim.  We make use of the following two technical lemmas for this purpose, whose proofs we postpone until Section \ref{s.lemproofs}.

\begin{lemma}
\label{l.V4exists}
Let $p_1,p_2,p_3 \in \R^2$ be in general position, and let $v_i=(p_j-p_k)^\perp$ be the perpendicular vector for which the ray $p_i+sv_i$ $(s\in \R^+)$ intersects the segment $\overline{p_jp_k}$, for each $\{i,j,k\}=\{1,2,3\}$. If $\phi_1,\phi_2,\phi_3$ are quadratic functions satisfying
\[
D^2\phi_i=A_i,\quad D\phi_i(p_k)=D\phi_j(p_k),\quad\mbox{and}\quad \phi_i(p_k)=\phi_j(p_k)
\]
for each $\{i,j,k\}=\{1,2,3\}$,
where $A_i=B^-+v_i\otimes v_i$ and $\trace(v_i\otimes v_i)>2(\trace(B)-2)$ for some matrix $B$ and vectors $v_i$ perpendicular to $p_j-p_k$ for each $\{i,j,k\}=\{1,2,3\}$, then there is a (unique) choice of $X_0\in \R^2$, $y_i=X_0+t_iv_i$ for $t_i/(v_i\cdot p_j)>1$, and $b \in \R^2, c \in \R$ such that the map 
\begin{equation*}
\phi_0(x) := \frac{1}{2} x^t B^- x +\frac 1 2 \trace(B)\abs{x-X_0}^2 +b^t x + c
\quad\mathrm{for}\quad x \in V_4
\end{equation*}
satisfies $\phi_0(y_i)=\phi_j(y_i)$ and $\phi_0'(y_i)=\phi_j'(y_i)$ for each $\{i,j\}\sbs \{1,2,3\}$.
\end{lemma}

\begin{lemma}
\label{rank1difference}
Suppose the points $p_1, p_2, p_3 \in \R^2$ are in general position and the quadratic functions $\varphi_1, \varphi_2, \varphi_3 : \R^2 \to \R$ satisfy
\begin{equation*}
\varphi_i(p_k) = \varphi_j(p_k) \quad \mbox{and} \quad D\varphi_i(p_k) = D\varphi_j(p_k),
\end{equation*}
for $\{ i, j, k \} = \{ 1, 2, 3 \}$.  There is a matrix $B$ and coefficients $\alpha_i \in \R$ such that
\begin{equation}
\label{l.rank1}
D^2 \varphi_i = B^- + \alpha_i (p_j - p_k)^\perp \otimes (p_j - p_k)^\perp,
\end{equation}
for $i = 1, 2, 3$.
\end{lemma}

Observe now that in the setting of the claim, the conditions of Lemma \ref{l.V4exists} are satisfied for $A_i:=B_i$ ($i=1,2,3$), $B:=B_0$ and where $v_i$ is the vector for which $B_i-B_0=v_i\otimes v_i$ for each $i=1,2,3$; indeed Observations \ref{o.vecsq} and the definition of Apollonian curve ensure that $v_i$ is perpendicular to $p_j-p_k$ for each $\{i,j,k\}=\{1,2,3\}$.  Let now $X_0$, $t_i$, and $y_i$ be as given in the Lemma.  We wish to show that $y_i=\gamma_i(\alpha(B_0))$ for each $i$.  Letting $B_{ij}$ denote the successor of $(B_0,B_i,B_j)$ for $\{i,j\}\sbs \{1,2,3\}$, we apply Lemma \ref{rank1difference} to the triples $\{p_i,p_j,y_k\}$ of points and $\{\phi_i,\phi_j,\phi_0\}$ of functions for each of the three pairs $\{i,j\}\sbs \{1,2,3\}$.  In each case, we are given some matrix $B$ for which 
\begin{align}
&\label{phii}B_i=B^-+\alpha_{k,s} (p_k - y_i)^\perp \otimes (p_k - y_i)^\perp,\\
&\label{phij}B_j=B^-+\alpha_{k,s} (p_k - y_j)^\perp \otimes (p_k - y_j)^\perp,\mbox{ and}\\
&\label{phi0}B_0=B^-+\alpha_{k,0} (y_i - y_j)^\perp \otimes (y_i - y_j)^\perp 
\end{align}
for real numbers $\alpha_{k,i}\in \R$.  

Observation \ref{o.succ} now implies that either $B=B_{ij}$ or $B=B_k$; the latter possibility cannot happen, however: if we had $B=B_k$, then as $\bar \rho(B_k-B_0)=-\bar \rho(B_0-B_k)$, Observation \ref{o.vecsq} would imply that $y_i-y_j$ is perpendicular to $p_i-p_j$.  This is impossible since the constraint $t_s/(v_s\cdot p_t)>1$ for $\{s,t\}=\{i,j\}$ in Lemma \ref{l.V4exists} implies that the segment $y_iy_j$ must intersect the segments $p_ip_k$ and $p_jp_k$, yet part \ref{i.csep} of Proposition \ref{p.bounds} implies that $\triangle p_ip_jp_k$ is acute.  So we have indeed that the matrix $B$ given by the applications of Lemma \ref{rank1difference} to the triple $(B_0,B_i,B_j)$ is $B_{ij}$, for each $\{i,j\}\sbs \{1,2,3\}$.

For each $\{i,j,k\}=\{1,2,3\}$, Observation \ref{o.vecsq}, Definition \ref{d.appcurve}, Theorem \ref{t.appcurve}, and the constraints \eqref{phii}, \eqref{phij} now imply that $y_i=q_i:=\gamma_i(\alpha_i(B_0))$, as the point $\gamma_i(\alpha_i(B_0))$ is determined by the endpoints $\gamma_i(\alpha_k(B_j)),\gamma_i(\alpha_i(B_k))$ and the condition from Definition \ref{d.appcurve} that $\gamma_i(\alpha_i(B_j))-\gamma_k(\alpha_i(B_0))$ and $\gamma_i(\alpha_i(B_k))-\gamma_i(\alpha_i(B_0))$ are multiples of $v_i^{1/2}(s_i(B_j,B_0))$ and $v_i^{1/2}(s_i(B_k,B_0))$, respectively (and so of $p_k-y_i$ and $p_k-y_j$, respectively, by \eqref{phii} and \eqref{phij}).

Similarly, the constraint \eqref{phi0} implies that $q_i-q_j$ is a multiple of $v_{0}^{1/2}(B_{ij})$ for the function $v^{1/2}_{0}$ defined with respect to the circle $B_0$.  Definition \ref{d.appcurve} and Theorem \ref{t.appcurve} now imply the existence of the curve $\gamma_k'$, completing the proof of the claim.

\subsection{Full measure}
\label{s.fm}
\begin{figure}[t]
\begin{center}
\includegraphics[width=.61\linewidth]{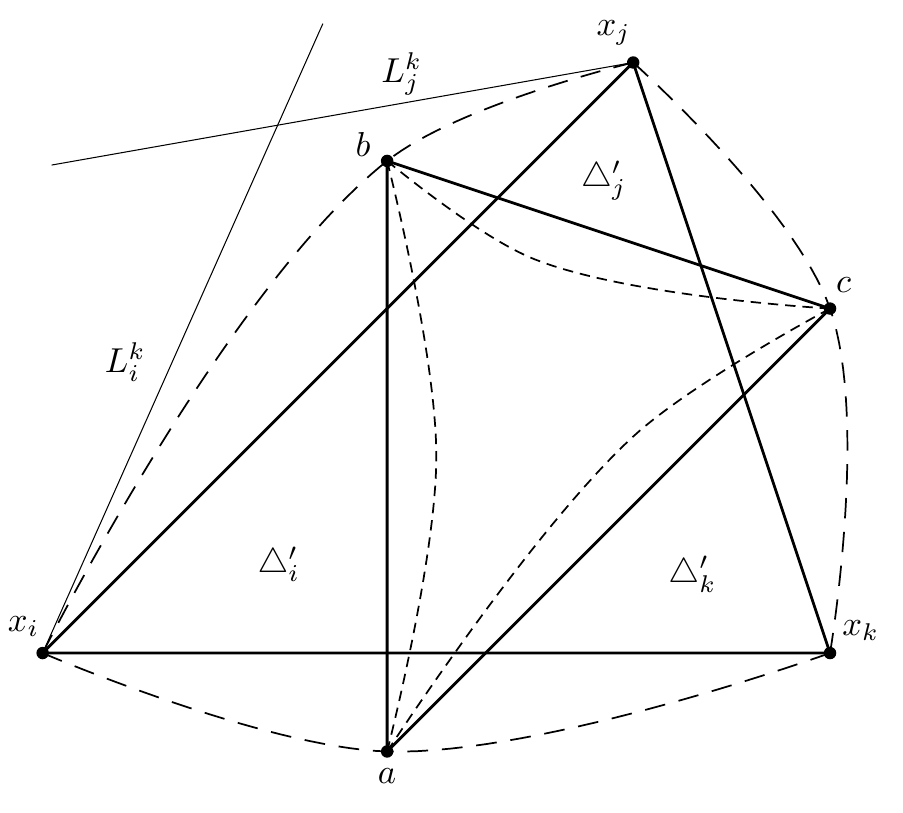}
\vspace{-1em}
\end{center}
\caption{\label{f.kappa}To show that $\sss$ has full measure in $Z$, we show that each Apollonian triangle $V_\ell$ has area which is a universal positive constant fraction of the area of the region $R_\ell$ it subdivides.  Here, the boundaries of $R_\ell$ and $V_\ell$ are shown in long- and short-dashed lines, respectively.}
\end{figure}

We begin by noting a simple fact about triangle geometry, easily deduced by applying a similarity transformation to the fixed case of $L=1$:
\begin{proposition}
\label{o.cd}
  Any angle $a$ determines constants $C_a,D_a$ such that any triangle $\tri$ which has an angle $\theta\geq a$ and opposite side length $\ell\leq L$ has area $A(\tri)\leq C_a L^2$, and any triangle which has angles $\theta_1\geq a,\theta_2\geq a$ sharing a side of length $\ell\geq L$ has area $\geq D_a L^2$.\qed
\end{proposition}

We wish to show that the interior of $\sss$ has full measure in $Z$, defined as the convex closure of $\sss$.   Recall that the straight sides of each pair of incident degenerate Apollonian triangles $V_i,V_j$ $(\{i,j\}\sbs \{1,2,3\}$ intersect at right angles, so the 6 straight sides of $V_1,V_2,V_3$ will form a convex boundary for $Z$.

Letting thus $Y_t=B\stm S_t$, we have that $Y_t$ is a disjoint union of some open sets $R_\ell$ bordered by three pairwise intersecting Apollonian triangles, and $X_{t+1}$ contains in each such region an Apollonian triangle $V_\ell$ dividing the region further.  To prove that the interior of $\sss$ has full measure in $Z$, it thus suffices to show that the area $A(V_\ell)$ is at least a universal positive constant fraction $\kappa$ of the area $A(R_\ell)$ for each $\ell$, giving then that $\mu(Y_t)\leq (1-\kappa)^{t-1} \mu(Y_1)\to_t 0$.

For $R_\ell$ bordered by Apollonian triangles $V_i,V_j,V_k$ and letting $\tri'=\tri x_ix_jx_k$ be the triangle whose vertices $x_s$ are the points of pairwise intersections $V_t,V_u$ for each $\{s,t,u\}=\{i,j,k\}$ of the Apollonian triangles bordering $R_\ell$, we will begin by noting that there is an absolute positive constant $\kappa'$ such that $A(\tri')\geq \kappa' \mu(R_\ell)$.  For each $\{s,t,u\}= \{i,j,k\}$, the segment $x_sx_t$ together with the lines $L^u_s$ and $L^u_t$ tangent to the boundary of $V_u$ at $x_s$ and $x_t$, respectively, form a triangle $\tri_u$ such that $R_\ell\sbs \tri'\cup \tri_i\cup \tri_j\cup \tri_k$.   Observations \ref{o.Aintangles}, \ref{o.Acntangles}, and \ref{o.cd} now imply that area of each $\tri_i$ is universally bounded relative to the area of $\tri'$, giving the existence $\kappa'$ satisfying $A(\tri')\geq \kappa'\mu(R_\ell)$.

It thus remains to show that the Apollonian triangle $V_\ell$ which subdivides $\ell$ satisfies $\mu(V_\ell)\geq \kappa'' A(\tri')$ for some $\kappa''$.  (It can in fact be shown that $\mu(V_\ell)=\frac{4}{21} A(\tri')$ exactly, but a lower bound suffices for our purposes.)  Considering the triangle $\tri''=\tri a b c$ whose vertices are the three vertices of $V_\ell$, there are three triangular components of $\tri'$ lying outside of  $\tri''$; denote them by $\tri'_i,\tri'_j,\tri'_k$ where $\tri'_s$ includes the vertex $x_s$ for each $s=i,j,k$.  The bound $\angle x_sx_tx_u>\frac \pi 4$ for each $\{s,t,u\}=\{i,j,k\}$ together with Observation \ref{o.Aintangles} implies there is a universal constant bounding the ratio of the area of $\tri'_s$ to $\tri''$ for each $s=i,j,k$.  Thus we have that the area of $\tri'$ is universally bounded by a positive constant fraction of the area of $\tri''$, and thus via Lemma \ref{L.47} we have that there is a universal constant $\kappa''$ such that $\mu(V_\ell)\geq \kappa'' A(\tri')$.

Taking $\kappa=\kappa'\cdot \kappa''$ we have that  $\mu(V_\ell)\geq \kappa \mu(R_\ell)$ for all $\ell$, as desired, giving that the measures $\mu(Y_t)$ satisfy $\mu(Y_t)=(1-\kappa)^{t-1}\mu(Y_1)\to 0$, so that $\mu(\sss)=\mu(Z)$.

\subsection{Proofs of two Lemmas}
\label{s.lemproofs}

\begin{proof}[Proof of Lemma~\ref{l.V4exists}]
We have 
\begin{equation*}
\phi_i(x) = A_ix +d_i =(B^-+v_i\otimes v_i)x+d_i
\end{equation*}
and the agreement of $D\phi_j,D\phi_k$ at $x_i$ $(\{i,j,k\}=\{1,2,3\})$ together with $v_i\cdot (x_j-x_k)=0$ implies that the $D\phi_i-B^-x$ is constant independent of $i$ on the triangle 
$\triangle x_1x_2x_3$, giving that
\begin{equation}
D\phi_i(x)= A_i x -(v_i\otimes v_i) x_j +d,
\label{l.uprime}
\end{equation}
for a constant $d\in \R^2$ independent of $i$.  Similarly, the value agreement constraints give that \begin{equation}
\phi_i(x)= \frac 1 2 x^tA_i x -\frac 1 2 x_j^t(v_i\otimes v_i) x_j +dx+c,
\label{l.u}
\end{equation}
for a constant $c$ independent of $i$.  Thus, by setting $D:=d$ adjusting $C$ in the definition of $u_1$ as necessary, we may assume that in fact $c$ and $d$ are 0.

Fixing any point $X_0$ inside the triangle $x_1x_2x_3$, we define for each $i$ a ray $R_i$ emanating from $X_0$ coincident with the line $\{tv_i\st t\in \R^+\}$.  Our goal is now to choose $X_0$ such that there are points $y_i$ on each of the rays $R_i$ satisfying the constraints of the Lemma.  

On each ray $R_i$, we can parameterize $\bar \phi_i:=\phi_i(x)-\frac 1 2 x^tB^-x$ as functions $f_i(t_i)=\frac 1 2 a_it^2$ $(i=1,2,3)$, and the function $\bar \phi_0:=\phi_0(x)-\frac 1 2 x^tB^-x$ as $g_i(t_i)=\frac 1 2 b(t+h_i)^2+C$, where $t_i$ is the distance from the line $\overline{x_jx_k}$ to $x\in R_i$, $h_i$ is the distance from $X_0$ to the line $\overline{x_jx_k}$, and $a_i$ and $b$ are $\trace(v_i\otimes v_i)$ and $\trace(B)$, respectively.  Moreover, since the gradients of $\bar \phi_i$ and $\bar \phi_0$ can both be expressed as multiple of $v_i$ along the whole ray $R_i$, we have for any point $x$ on $R_i\cap U_i$ at distance $t$ from $\overline{x_jx_k}$ that $f_i'(t_i)=g_i'(t_i)$ implies that $D\phi_i(x)=D\phi_0(x).$  Thus to prove the Lemma, it suffices to show that there are $X_0$ and $C$ such that for the resulting values of $h_i$, the systems
\begin{equation*}
\left\{
\begin{array}{ll}
f_i(t_i)=g_i(t_i)\\
f_i'(t_i)=g_i'(t_i)
\end{array}
\right.
\quad 
\mbox{or, more explicitly, }
\quad
\left\{
\begin{array}{ll}
\frac 1 2 a_it_i^2=\frac 1 2 b(t_i+h_i)^2+C\\
a_it_i=b(t_i+h_i)
\end{array}
\right.
\end{equation*}
have a solution over the real numbers for each $i$.

It is now easy to solve these systems in terms of $C$; for each $i$, 
\[
 t_i=\frac{bh_i}{a_i-b}  \quad \mbox{and}\quad h_i=\frac{\sqrt{-C}}{\sqrt{\frac 1 2 \left(b- \frac{b^2}{a_i-b}\right)}}
\]
gives the unique solution.   Note that $a_i>2b$ ensures that the denominator in the expressions for $h_i$ (and $t_i$) is positive for each $i$.  Since $\sqrt{-C}$ takes on all positive real numbers and $\trace(A_i)=a_i-b$, there is a (negative) value $C$ for which the distances $h_i$ are the distances from the lines $x_jx_k$ to a point $X_0$ inside $\triangle x_1x_2x_3$; it is the point with trilinear coordinates 
$
\left\{ \left(\frac{\trace(A_i)-\trace(B)}{\trace(A_i)}\right)^{-\frac 1 2}\right\}_{1\leq i \leq 3}.
$

 The Lemma is now satisfied for this choice of $C$ and $X_0$ and for the points $y_i$ on $R_i$ at distance $h_i+t_i$ from $X_0$ for $i=1,2,3$.
\end{proof}

\begin{proof}[Proof of Lemma~\ref{rank1difference}]
Let $q_1 = p_3 - p_2$, $q_2 = p_1 - p_3$, and $q_3 = p_2 - p_1$ and $A_i := D^2 \varphi_i$. Since for any individual $i=1,2,3$ we could assume without loss of generality that $\varphi_i \equiv 0$, the compatibility conditions with $\phi_j$, $\phi_k$ give
\begin{equation}
\label{gradcon}
\begin{cases}
(A_j - A_i)q_j + (A_k - A_i)q_k = 0, \\
q_j^t (A_j - A_i) q_j - q_k^t (A_k - A_i) q_k = 0
\end{cases}
\end{equation}
in each case.
If we left multiply the first by $q_k^t$ and add it to the second, we obtain
\begin{equation}
\label{crosscon}
q_i^t(A_j - A_i)q_j = 0.
\end{equation}
Since $q_i \cdot q_j \neq 0$, there are unique $\alpha_{ij}, \beta_{ij}, \gamma_{ij} \in \R$ such that
\begin{equation*}
A_j - A_i = \alpha_{ji} q_j^\perp \otimes q_j^\perp + \beta_{ji} q_i^\perp \otimes q_i^\perp + \gamma_{ij} (q_i^\perp \otimes q_j^\perp + q_j^\perp \otimes q_i^\perp), 
\end{equation*}
where $(x,y)^\perp = (-y,x)$.  Since symmetry implies $\beta_{ji} = - \alpha_{ij}$ and \eqref{crosscon} implies $\gamma_{ij} = 0$, we in fact have
\begin{equation*}
A_j - A_i = \alpha_{ji} q_j^\perp \otimes q_j^\perp - \alpha_{ij} q_i^\perp \otimes q_i^\perp.
\end{equation*}
If we substitute this into \eqref{gradcon}, we obtain
\begin{equation*}
- \alpha_{ij} (q_i^\perp \cdot q_j) q_i - \alpha_{ik} (q_i^\perp \cdot q_k) q_i = 0.
\end{equation*}
Since $q_i^\perp \cdot q_j = - q_i^\perp \cdot q_k$, we obtain $\alpha_{ij} = \alpha_{ik}$. Thus there are $\alpha_i \in \R$ such that
\begin{equation*}
A_i - A_j = \alpha_i q_i^\perp \otimes q_i^\perp - \alpha_j q_j^\perp \otimes q_j^\perp.
\end{equation*}
In particular, we see that $A_i - \alpha_i q_i^\perp \otimes q_i^\perp$ is constant.
\end{proof}

\subsection{Proof of Corollary \ref{c.match}}
This is now an easy consequence of Theorem \ref{t.pq} and the viscosity theory, via Proposition \ref{p.visc}.
\begin{proof}[Proof of Corollary \ref{c.match}]
Write $v=v_\infty$. Continuity of the derivative and value of $v$ in $U_1 \cup U_2\cup U_3$ imply that 
\[
v|_{U_i}=\frac 1 2 x^tA_i x+Dx +C \quad \mbox{for}\quad i=1,2,3
\]
for some $D\in \R^2$, $C\in \R$.  Let $\beta_i$ be the portion of the boundary of $R$ between $x_j$ and $x_k$ which does not include $x_i$, and let $v_i$ be the vector perpendicular to $x_j-x_k$ such that $x_i+tv_i$ intersects the segment $x_jx_k$.

We let $V_i=\beta_i+tv_i$ for $t\geq 0$.  The $V_i$'s are pairwise disjoint. Thus, by first restricting the quadratic pieces $U_1,U_2,U_3$ of the map $v$ to their intersection with the respective sets $V_i$, and then extending the quadratic pieces to the full $V_i$'s, we may assume that $U_i=V_i$ for each $i=1,2,3$.

We apply Lemma \ref{rank1difference} to $v|_{V_1}$, $v|_{V_2}$, $v|_{V_3}$; by Observation \ref{o.succ} there are up to two possibilities for the matrix $B$ from \eqref{l.rank1}; the fact that $\triangle x_1x_2x_3$ is a acute, however, implies that we have that $B$ is the successor of $A_1,A_2,A_3$.  Thus letting $S$ denote the Apollonian triangulation determined by $x_1,x_2,x_3$, Theorem \ref{t.pq} ensures the existence of a $C^{1,1}$ map $u$ which is piecewise quadratic whose quadratic pieces have domains forming the Apollonian triangulation $\sss$ determined by the vertices $x_1,x_2,x_3$.  Letting $U_i'$ denote the degenerate Apollonian triangle in $\sss$ intersecting $x_j$ and $x_k$ for each $\{i,j,k\}=\{1,2,3\}$, we can extend $u$ to a map $\bar u$ by extending the three degenerate pieces $U'_i$ of $\sss$ to sets $V_i'=\{x+tv_i\st x\in U_i',t\geq 0\}$.  Now we can find curves $\gamma_i$ from $x_j$ to $x_k$ lying inside $V_i\cap V_i'$, and, letting $\Omega$ be the open region bounded by the curves $\gamma_1,\gamma_2,\gamma_3$, Proposition \ref{p.visc} implies that $\bar u+Dx+C$ and $v$ are equal in $\Omega$, as they agree on the boundary $\partial \Omega=\gamma_1\cup \gamma_2\cup \gamma_3$.
\end{proof}

\section{Further Questions}
\begin{figure}
\includegraphics[width=.9\linewidth]{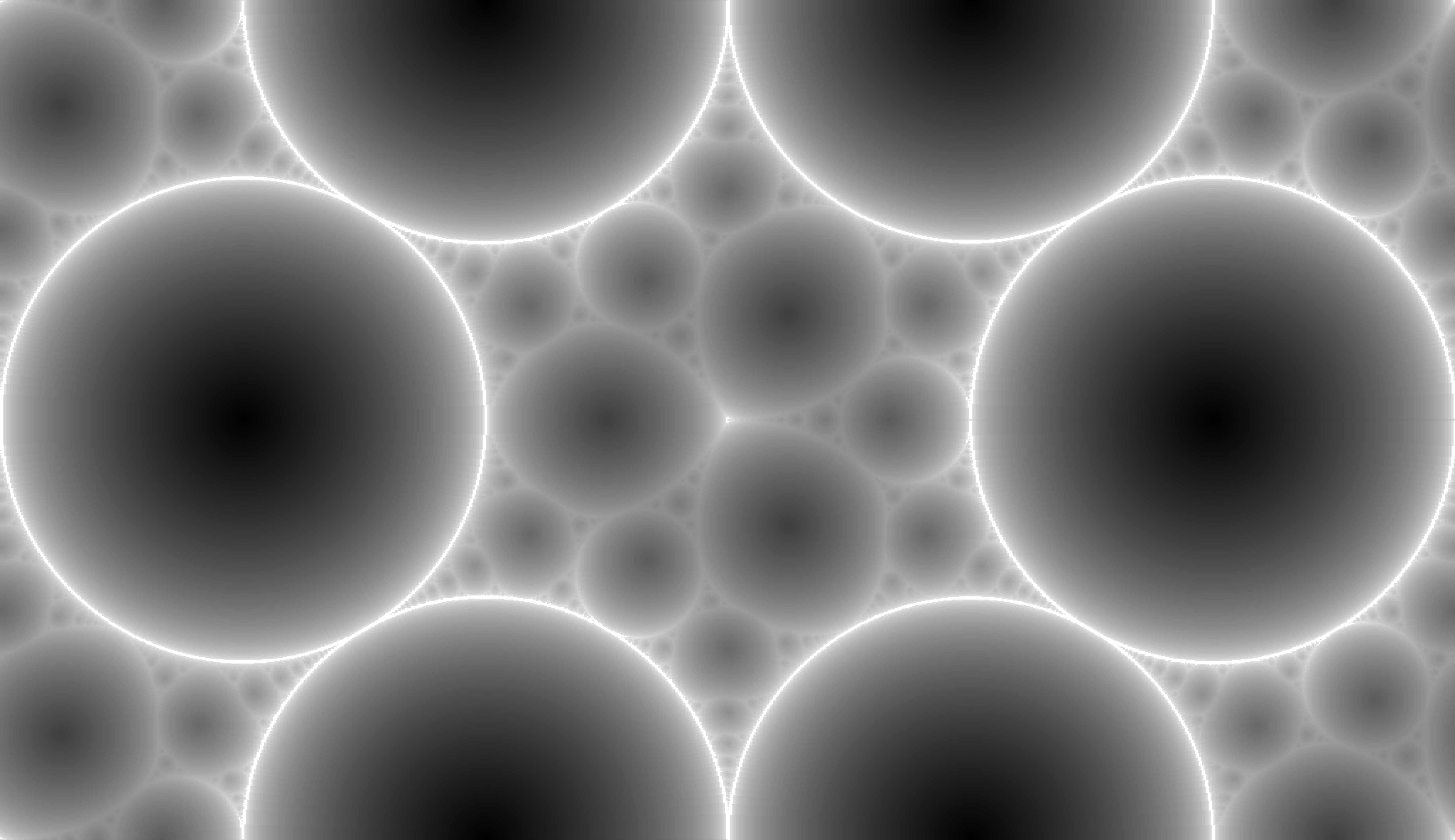}
\caption{\label{f.gammatri} The graph of the function $c \in C(\R^2)$ over the rectangle $[0,6] \times [0,6/\sqrt{3}]$, where $c$ describes the boundary \[\partial \Gamma_{\rm tri} = \{ \mbox{$\frac 23$} M(a,b,c(a,b)) : a, b \in \R \}\] of $\Gamma_{\rm tri}$.  White and black correspond to $c = 3$ and $c = 4$, respectively.}
\end{figure}

\label{s.Q}
Our results suggest a number of interesting questions. To highlight just a few, one direction comes from the natural extension of both the sandpile dynamics and the definition of $\Gamma$ to other lattices.

\begin{question}\label{q.tri}While the companion paper \cite{2014integersuperharmonic} determines $\Gamma(\Z^2)$, the analogous set $\Gamma(\LLL)$ of stabilizable matrices for other lattices $\LLL$ is an intriguing open problem.
For example, for the triangular lattice $\LLL_{\rm tri} \subseteq \R^2$ generated by $(1,0)$ and $(1/2,\sqrt 3 / 2)$,  the set $\Gamma(\LLL_{\rm{tri}})$ is the set of $2\times 2$ real symmetric matrices $A$ such that there exists $u:\LLL_{\rm tri}\to \Z$ satisfying
\begin{equation}
u\geq \frac 1 2 x^tAx\quad\mbox{and}\quad\Delta u\leq 5,
\end{equation}
where here $\Delta$ is the graph Laplacian on the lattice.
The algorithm from Section \ref{s.gamma} can be adapted to this case and we display its output in Figure~\ref{f.gammatri}.  While the Apollonian structure of the rectangular case is missing, there does seem to be a set $\ppp_{\rm tri}$ of isolated ``peaks'' such that $\bar \Gamma_{\rm tri} = \ppp_{\rm tri}^\downarrow$.  What is the structure of these peaks?  What about other lattices or graphs?  
Large-scale images of $\Gamma(\LLL)$ for other planar lattices $\LLL$ and the associated sandpiles on $G$ can be found at \cite{Wes}.
\end{question}

Although we have explored several aspects of the geometry of Apollonian triangulations, many natural questions remain. For example: 
\begin{question}
Is there a closed-form characterization of Apollonian curves?  
\end{question} 
\noindent Apollonian triangulations themselves present some obvious targets, such as the determination of their Hausdorff dimension.

\end{document}